\documentclass[12pt,reqno]{amsproc}

\usepackage[margin=.9in]{geometry}

\usepackage[T1]{fontenc}
\usepackage{bbm}
\usepackage{ifsym}
\usepackage[usenames]{color}
\usepackage{mathpazo}
\usepackage{pict2e} 
\usepackage{amssymb,amsthm,nicefrac}
\usepackage{wrapfig}
\usepackage{tikz-cd}
\usepackage{tkz-graph}
\usetikzlibrary{positioning, fit, patterns, arrows.meta}
\usepackage[linesnumbered,ruled,vlined]{algorithm2e} 
\SetKwInput{KwInput}{Input}                
\SetKwInput{KwOutput}{Output}              
\usepackage{caption}
\usepackage{mathrsfs,bm}
\usepackage[all,cmtip]{xy}
\usepackage[normalem]{ulem}

\makeatletter
\DeclareRobustCommand{\bigtimes}{%
	\mathop{\vphantom{\sum}\mathpalette\@bigtimes\relax}\slimits@
}
\newcommand{\@bigtimes}[2]{\vcenter{\hbox{\make@bigtimes{#1}}}}
\newcommand{\make@bigtimes}[1]{%
	\sbox\z@{$\m@th#1\sum$}%
	\setlength{\unitlength}{\wd\z@}%
	\begin{picture}(1,1)
	\roundcap
	\linethickness{.17ex}
	\Line(0.01,0)(.85,.99)
	\Line(0.01,.99)(.85,0.0)
	\end{picture}%
}
\makeatother

\definecolor{PineGreen}{rgb}{0.0,0.47,0.44}
\definecolor{MidnightBlue}{rgb}{0.1,0.1,0.44}
\definecolor{magenta}{rgb}{1.0,0.0,1.0}
\definecolor{bl1}{HTML}{4479A1}
\definecolor{pur1}{HTML}{52196D}
\definecolor{mag1}{HTML}{2AD0F1}
\definecolor{org1}{rgb}{.92,.39.21}
\definecolor{pur2}{rgb}{.53,.47,.7}

\usepackage{hyperref}
\hypersetup{
	colorlinks=true,
	linkcolor=blue,
	citecolor=blue,
	filecolor=magenta,
	urlcolor=blue
}

\makeatletter
\newcommand{\eqnum}{\refstepcounter{equation}\textup{\tagform@{\theequation}}}
\makeatother

\newtheorem{theorem}{Theorem}
\newtheorem*{theorem*}{Theorem}

\numberwithin{theorem}{section}
\newtheorem{proposition}[theorem]{Proposition}
\newtheorem{lemma}[theorem]{Lemma}
\newtheorem{corollary}[theorem]{Corollary}
\theoremstyle{definition}

\newtheorem{definition}[theorem]{Definition}

\theoremstyle{remark}
\newtheorem{remark}[theorem]{Remark}
\newtheorem{example}[theorem]{Example}

\newcommand{\RR}{\mathbb{R}}
\newcommand{\R}{\mathbb{R}}

\def \tr{{\rm tr}}

\renewcommand{\phi}{\varphi}

\newcommand{\cF}{\mathscr{F}}
\newcommand{\T}{^\mathsf{T}}

\newcommand{\bA}{\mathbf{A}}
\newcommand{\PC}{\text{\bf PC}}

\newcommand{\Eq}{\text{\rm Eq}}
\newcommand{\MSC}{\text{\bf MSC}}

\newcommand{\tot}{\text{\rm Tot}}

\newcommand{\set}[1]{{\left\{{#1}\right\}}}

\newcommand{\inj}{\hookrightarrow}

\begin{document}
	
	\title{Principal Components along Quiver Representations} 
	\author{Anna Seigal}
	\author{Heather A. Harrington}
    \author{Vidit Nanda}
	
	\maketitle
	
	\begin{abstract}
		Quiver representations arise naturally in many areas across mathematics. Here we describe an algorithm for calculating the vector space of sections, or compatible assignments of vectors to vertices, of any finite-dimensional representation of a finite quiver. Consequently, we are able to define and compute principal components with respect to quiver representations. These principal components are solutions to constrained optimisation problems defined over the space of sections, and are eigenvectors of an associated matrix pencil.
	\end{abstract}

	\section*{Introduction}
	
	A {\bf quiver representation} is an arrangement of vector spaces and linear maps tethered to the vertices and edges of a directed graph \cite{derksen2017introduction, schiffler2014quiver}. The quiver illustrated below will be our running example throughout the paper.
	
	\begin{center}
	\includegraphics[width=.7\textwidth]{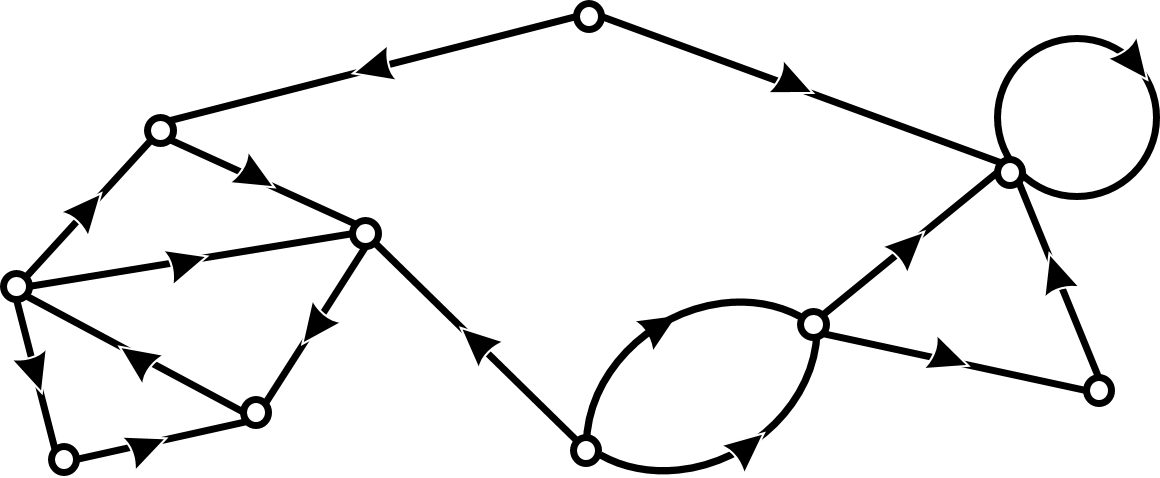}
	\end{center}
	
	\noindent Despite being relatively concrete mathematical objects, quiver representations provide a uniform framework for a host of fundamental abstract problems in linear algebra \cite{bernstein1973coxeter}.
	Isomorphisms of quiver representations can be used to characterise, for example, the Jordan normal form of matrices and the Kronecker normal form of matrix pencils.
	They also play an important role in various other fields, including the study of associative algebras \cite{auslander1997representation}, Gromov-Witten invariants \cite{gross2010quivers}, representations of Kac-Moody algebras \cite{nakajima2001quiver}, moduli stacks \cite{Toda2018moduli}, Morse theory \cite{kirwan2019morse}, persistent homology \cite{oudot2015persistence}, and perverse sheaves \cite{gelfand1996perv}, among others. 
	
	In most of these contexts, the crucial property of a given quiver representation is its {\em decomposability} into a direct sum of smaller representations. Gabriel's celebrated result \cite{gabriel1972quiver} establishes that a quiver admits finitely many (isomorphism classes of) indecomposable representations if and only if its underlying undirected graph is a union of simply-laced Dynkin diagram (i.e., type $A$, $D$ or $E$). Thus, most quivers have rather complicated sets of indecomposable representations, and are said to be of {\em wild} type. It is a direct consequence of this trifecta -- concreteness, ubiquity and generic wildness -- that ideas from disparate branches of mathematics have conversely been deployed to study representations of quivers. These include algebraic geometry \cite{kirillov2016quiver}, combinatorics \cite{derksen2011combinatorics}, differential geometry \cite{harada2011Morse, haiden2017semistability}, geometric representation theory \cite{ginzburg2009lectures}, invariant theory \cite{king1994moduli, kac1982infinite}, and multilinear algebra \cite{herschend2008tensor}. 
	
    Quiver representations have recently emerged in far more applied and computational contexts than the classical ones listed above. We are aware of three such appearances: 
    \begin{enumerate} 
    \item {\bf Cellular Sheaves:} A vector space valued sheaf defined over a cell complex \cite{curry2013sheaves, curry2016discrete} constitutes a representation of the underlying {\em Hasse diagram}; here the vertices are cells and edges arise from face inclusions. The stalks of the sheaf form vector spaces over the vertices while  restriction maps are associated to edges. 
    \item {\bf Conley theory:} {Morse decompositions} in computational dynamics \cite[Def 9.19]{harker2018computational} are representations of Conley-Morse quivers associated to discrete dynamical systems (vertices are recurrent sets and edges represent gradient flow). The linear maps of such representations arise from {\em connection matrices} \cite{franzosa1989connection}; these assemble into a chain complex that allows one to recover the homology of the phase space.
    \item {\bf Algebraic statistics:} Matrix normal models can be studied via quiver representations \cite{amendola2021invariant,derksen2021maximum}.
    The sample data gives a representation of a Kronecker quiver. The stability of the representation~\cite{king1994moduli} can then be used to characterise the existence and uniqueness of a maximum likelihood estimate in the model.
    \end{enumerate}
    \noindent We expect (and hope) that this influx of quiver theory into more applied and computational domains will continue. 
    
    \subsection*{This Paper} We consider a representation $\bA_\bullet$ of a quiver $Q$. It assigns vector spaces $\bA_v$ to each vertex $v$ of $Q$ and linear maps $\bA_e$ to each edge $e$ in $Q$. We construct a vector space $\Gamma(Q;\bA_\bullet)$ called the {\bf space of sections} of the quiver representation. An element of $\Gamma(Q;\bA_\bullet)$ selects one vector $\gamma_v$ from the vector space $\bA_v$ assigned to each vertex $v$ so that for every edge $e:u \to v$ the linear map $\bA_e$ sends $\gamma_u$ to $\gamma_v$. As such, $\Gamma(Q;\bA_\bullet)$ is a subspace of the {\em total space}  $\tot(\bA_\bullet) := \prod_v \bA_v$. The assignment 
    \[
    \bA_\bullet \mapsto \Gamma(Q;\bA_\bullet)
    \] can directly be seen to be a functor from the category of $Q$-representations to the category of vector spaces. We do not expect this functor to immediately answer any deep questions regarding (in)decomposability of quiver representations. Rather, we hope that the space of sections will become a useful and practical tool for those who encounter quiver representations in applied and computational contexts.

    Our first contribution is an algorithm for computing the space of sections for any finite-dimensional representation of a finite quiver. This is of some relevance even to those who have no warm feelings for quiver representations, since it is a purely categorical procedure for computing the limit (i.e., the universal cone) of a diagram in the category of vector spaces. With minor modifications, it can be made to work for diagrams valued in any abelian category that has computable products and equalisers. There are two 
    types of restriction imposed on
    the space of sections: the first of these arises from directed cycles, where we are forced to restrict to a fixed point space of an endomorphism; and the second is the presence of multiple incoming edges at a vertex, where we are forced to restrict to an equaliser. 
        None of these difficulties arise when the quiver is a directed rooted tree.
        
        Our algorithm consists of two steps --- the first step removes all directed cycles and updates the representation $\bA_\bullet$ accordingly; and the second step replaces this acyclic quiver with a directed rooted tree, again updating the representation. The result is a new representation $\bA^+_\bullet$ of a rooted directed tree $T^+$, which has all the same vertices as $Q$ (plus an additional root vertex) and satisfies $\bA^+_v \subset \bA_v$ at each vertex $v$. 
    
    Here is our first main result.
    
    \begin{theorem*} [A] The space of sections $\Gamma(Q;\bA_\bullet)$ is the image of the map 
    \[
    F:\bA^+_\rho \longrightarrow \tot(\bA_\bullet),
    \] obtained by composing the linear maps assigned by the quiver representation $\bA^+_\bullet$ along the unique path in the rooted directed tree $T^+$ from the root $\rho$ to each other vertex.
    \end{theorem*}
    
   \noindent Although the constructions of $T^+$ and $\bA_\bullet^+$ are explicit and readily implementable on a computer, they require making several intermediate choices. Each such choice is liable to produce a different $F$, but its image is always $\Gamma(Q;\bA_\bullet)$ regardless of these choices.
    
   Our second contribution takes place in the realm of quiver representations valued in real vector spaces; in this case, a map $F$ as described in Theorem (A) can be represented by an $n \times d$ full-rank real matrix, where $n$ and $d$ are the dimensions of $\tot(\bA_\bullet)$ and $\Gamma(Q;\bA_\bullet)$ respectively. Using this matrix, we define the {\bf principal components} of any (generic, mean-centered) finite set $D$ of vectors in $\R^n \simeq \tot(\bA_\bullet)$ with respect to the quiver representation~$\bA_\bullet$. As with ordinary principal components, the starting point is the $n \times n$ sample covariance matrix $S$ of the vectors in $D$. Next, we consider for each $r \leq d$ the variational problem of maximising the trace $\tr(X\T SX)$ over the set of all $n \times r$ matrices $X$ that satisfy $X\T X = \text{id}$, and whose columns are constrained to lie in the image of $F$, i.e., in $\Gamma(Q;\bA_\bullet)$. 
   There is a {generically} unique solution, obtained by iteratively incrementing $r$ from $1$ to $d$, and the span of its $r$-th column is the $r$-th principal component of $D$ along $\bA_\bullet$, denoted $\PC_r(D;\bA_\bullet)$. Unlike ordinary principal components, the $\PC_r(D;\bA_\bullet)$ are not spanned by eigenvectors of $S$ in general. 
   
   The second main result of this paper is that the quiver principal components $\PC_r(D;\bA_\bullet)$ do in fact admit a spectral interpretation. 
   
   \begin{theorem*}[B] For each $1 \leq r \leq d$, the $r$-th principal component $\PC_r(D;\bA_\bullet)$ is spanned by $Fu_r$, where $u_r$ is the eigenvector of the matrix pencil $F\T S F - \lambda(F\T F)$ corresponding to its $r$-th largest eigenvalue.
   \end{theorem*}
 
We see from the matrix pencil in Theorem (B) that the principal components along a quiver representation intertwine the properties of $D$ (via the sample covariance matris $S$) with those of quiver $Q$ (via the map $F$ to the space of sections). These principal components find directions of maximum variation among vectors in $D \subset \RR^n$ that respect certain linear dependencies. The coordinates in $\RR^n$ can be thought of as partitioned into blocks (one per vertex of the quiver); these blocks are related by the linear maps of the quiver representation. In this way, the principal components along the quiver representation interpolate between concatenating ordinary principal components from individual blocks, and ordinary principal components in the whole space $\RR^n$. We will mostly assume that the linear maps are fixed in advance, but we will briefly discuss approaches to learning them from the set $D$. 
   
   \subsection*{Related Work} 
   The first half of this work is inspired by the study of cellular sheaves \cite{curry2013sheaves}, which functorially assign vector spaces to cells and linear maps to incidence relations in a finite cell complex. The space of sections of a cellular sheaf $\mathscr{S}$ defined over an undirected graph $G$ is isomorphic to the zeroth sheaf cohomology group $\textbf{H}^0(G;\mathscr{S})$, which is readily computable \cite{curry2016discrete}. We can turn any representation $\bA_\bullet$ of a quiver $Q$ into a cellular sheaf over the underlying undirected graph by  replacing each edge-indexed linear~map \[\bA_u \stackrel{\bA_e}{\longrightarrow} \bA_v\] by a corresponding zigzag of the form \[\bA_u \stackrel{\bA_e}{\longrightarrow} \bA_v \stackrel{\text{id}}{\longleftarrow} \bA_v.\] Thus, each edge inherits the vector space assigned to its target vertex. Computing zeroth cohomology of this sheaf furnishes an alternative to Theorem (A) for calculating $\Gamma(Q;\bA_\bullet)$. However, this cohomological alternative suffers from two significant drawbacks --- first, the insertion of these zigzags is quite inconvenient for our purposes of testing compatibility of sections across directed paths in the original quiver. And second, the duplication of vector spaces over the edges leads to unnecessarily large matrices, and hence incurs a larger computational cost.
   
   A central focus of the second half of this paper is the study of linearly constrained principal components, which dates back at least to~\cite[Section 11]{rao1964use}. It is referred to as {\em constrained PCA} in~\cite[Section 7.1]{diamantaras1996principal} and~\cite{takane2001constrained, takane1991principal}. Its statistical implications are discussed in~\cite{hunter2002constrained} and \cite[Section 5.4]{takane2001constrained}. For an example of constrained PCA occurring in a biological context, see~\cite{huang2020statistical}. It is important to note that the principal components $\PC_r(D;\bA_\bullet)$ introduced in this paper do not constitute a low-rank approximation of the representation $\bA_\bullet$. Such approximation of related multi-linear objects appears in ~\cite{brake2019singular}, where the authors find the singular value decomposition of a finite chain complex, and in the study of orthogonal decomposition of tensor networks~\cite{halaseh2021orthogonal}. 
   A study of star quivers for parameter estimation in integrated PCA~\cite{tang2021integrated} appears in \cite{franks2021IPCA}.
   
   We comment on connections to linear neural networks in Section~\ref{sec:learn}. Quiver representations appear in the context of 
    neural network architectures in \cite{armenta2020representation, jeffreys2021kahler}, though these are not usual quiver representations due to the presence of nonlinear activation functions.
 
    \subsection*{Organisation} The remainder of this paper is divided into eight short sections. In \S \ref{sec:qreps}, we define quiver representations, their sections, and some elementary properties thereof. 
    
    \S \ref{sec:scq} is devoted to the task of using the ear decomposition to compute the sections of strongly-connected quivers. \S \ref{sec:acyclic} uses the results of \S \ref{sec:scq} to construct, from any given quiver representation, a sub-representation of an acyclic subquiver that has the same space of sections. In \S \ref{sec:arboreal} we describe how to further modify this acyclic subquiver into a rooted directed tree and update the overlaid representation to preserve the space of sections. These intermediate results are assembled in \S \ref{sec:dimension} to provide a proof of Theorem (A); we also give lower bounds on the dimension of the space of sections, and provide pseudocode for our algorithm along with a computational complexity analysis.
    
    Principal components along quiver representations are defined in \S \ref{sec:PCs} via three optimisation problems; we show here that all three give the same answer. In \S \ref{sec:gsvd} we use a generalisation of the singular value decomposition to establish Theorem (B). And finally, \S \ref{sec:learn} discusses the problem of learning the linear maps of a quiver representation from finite samples of vectors living in the total space.
    
    {\footnotesize
    \subsection*{Acknowledgements}
AS and HAH thank Julian Knight, and AS thanks Visu Makam, for helpful discussions. VN is grateful to Frances Kirwan for timely quiver-theoretic advice. 
We thank Darij Grinberg for helpful comments on the first version of this paper.
We thank the anonymous referees and editor for useful suggestions that improved the paper.
AS, HAH, and VN are members of the UK Centre for Topological Data Analysis and thank the Mathematical Institute at the University of Oxford. HAH and VN acknowledge funding from EPSRC grant EP/R018472/1. HAH gratefully acknowledges funding from EPSRC EP/R005125/1 \& EP/T001968/1, the Royal Society RGF$\backslash$EA$\backslash$201074 and UF150238, and Emerson Collective. }

	\section{Quiver Representations and Sections}\label{sec:qreps}
	
	A {quiver} $Q$ consists of a finite set $V$ whose elements are called {vertices}, a finite set $E$ whose elements are called {edges}, and two maps $s,t:E \to V$ called the {source} and {target} map respectively. It is customary to illustrate quivers by drawing points for vertices and arrows (from source to target) for edges. A {path} in $Q$ is an ordered finite sequence of distinct edges $p = (e_1,e_2,\ldots,e_k)$ with disjoint sources (i.e., $s(e_i) \neq s(e_j)$ when $i \neq j$) so that $s(e_{i+1})=t(e_i)$ holds for every $1 \leq i < k$:
	\[
	\xymatrixcolsep{.6in}
	\xymatrix{
		\bullet \ar@{->}[r]^{e_1} & \bullet \ar@{->}[r]^{e_2} & \cdots \ar@{->}[r]^{e_{k-1}} & \bullet \ar@{->}[r]^{e_k} & \bullet
	}
	\]
	The source and target maps extend from edges to paths via $s(p) = s(e_1)$ and $t(p) = t(e_k)$. We call $p$ a {\em cycle} if $s(p) = t(p)$, and call $Q$ {\em acyclic} if it does not admit any cycles.
	
	A {\bf representation} of $Q$ comprises an assignment $\bA_\bullet$ of a finite-dimensional vector space $\bA_v$ to every vertex $v$ in $V$ and a linear map $\bA_e:\bA_{s(e)} \to \bA_{t(e)}$ to every edge $e$ in $E$. We will remain agnostic to the choice of underlying field until Section \ref{sec:PCs}. Using the data of $\bA_\bullet$, one can associate to each path $p = (e_1,\ldots,e_k)$ the map $\bA_p:\bA_{s(p)} \to \bA_{t(p)}$ via
	\begin{equation}
	\label{eqn:pathp}
	\bA_p := \bA_{e_k} \circ \bA_{e_{k-1}} \circ \cdots \circ \bA_{e_2} \circ \bA_{e_1}.
	\end{equation}
	The {\em total space} of $\bA_\bullet$ is the direct product 
	\[
	\tot(\bA_\bullet) := \prod_{v \in V} \bA_v.
	\]
	The following terminology has been borrowed from analogous notions that arise in the study of sheaves and vector bundles.
	
	\begin{definition}\label{def:section}
		Let $\bA_\bullet$ be a representation of a quiver $Q = (s,t:E \to V)$. A {\bf section} of $\bA_\bullet$ is an element $\gamma = \set{\gamma_v \in \bA_v \mid v \in V}$ in $\tot(\bA_\bullet)$ satisfying the {\em compatibility} requirement $\gamma_{t(e)} = \bA_e (\gamma_{s(e)})$ across each edge $e$ in $E$.    
	\end{definition}

	The set of all sections of $\bA_\bullet$ is a vector subspace of $\tot(\bA_\bullet)$, which we denote by $\Gamma(Q;\bA_\bullet)$.
	The explicit computation of the space of sections $\Gamma(Q;\bA_\bullet)$, for any quiver $Q$ and representation $\bA_\bullet$, is one of the central objectives of this work. 
	
	\begin{remark}
	The product of general linear groups $\text{G} = \prod_v \text{GL}(\bA_v)$ acts on $\bA_\bullet$ by change of basis: given any $g = \set{g_v \in \text{GL}(\bA_v) \mid v \in V}$, the new representation $g\bA_\bullet$ assigns
	\[
	(g\bA)_v = \bA_v \qquad \text{and} \qquad (g\bA)_e = g_{t(e)} \circ \bA_e \circ g_{s(e)}^{-1}.
	\]
	This action descends to the space of sections via $\gamma \mapsto g\gamma$, where $(g\gamma)_v = g_v \gamma_v$, and so we have an isomorphism $	\Gamma(Q;\bA_\bullet) \simeq \Gamma(Q;g\bA_\bullet)$
	for every $g \in \text{G}$. In fact, a purely formal argument shows that the assignment $\bA_\bullet \mapsto \Gamma(Q;\bA_\bullet)$ is a functor from the category of representations of a fixed $Q$ to the category of vector spaces. Here a morphism $\cF_\bullet:\bA_\bullet \to \bA'_\bullet$ of $Q$-representations is a collection of $V$-indexed linear maps $\cF_v:\bA_v \to \bA'_v$ which commute with the edge-maps, i.e., for each $e \in E$ we have 
	\begin{equation}
	    \label{eqn:functor}
	    \bA'_e \circ \cF_{s(e)}  = \cF_{t(e)} \circ \bA_e.
	\end{equation}
Each section $\gamma \in \Gamma(Q;\bA)$ is sent by $\cF_\bullet$ to a section $\cF\gamma$ of $\bA'_\bullet$ prescribed by $(\cF\gamma)_v = \cF_v(\gamma_v)$, since applying~\eqref{eqn:functor} to $\gamma_{s(e)}$ gives the desired compatibility across each edge $e$:
\[
\bA'_e \circ \cF_{s(e)}(\gamma_{s(e)}) = \cF_{t(e)} \circ \bA_e (\gamma_{s(e)}) = \cF_{t(e)} \gamma_{t(e)}.
\]	\end{remark}

	Compatibility across edges imposes severe constraints on sections, even in the simplest of examples, when the underlying quiver $Q$ contains cycles or vertices with multiple incoming edges. 

	\begin{example}\label{ex:jordan}
		Consider the quiver that consists of a single vertex $v$ and a single edge $e$ with $s(e) = v = t(e)$. The space of sections of any representation $\bA_\bullet$ is the subspace of $\bA_v$ fixed by $\bA_e$, i.e., the eigenspace corresponding to eigenvalue $1$. 
	\end{example}
	
\begin{example}\label{ex:kronecker} The space of sections of a representation $\bA_\bullet$ of the $2$-Kronecker quiver, pictured below, is isomorphic to $\ker(\bA_e - \bA_f)$.
	\[
	\xymatrixcolsep{.6in}
	\xymatrix{
		u \ar@/^1.5pc/[r]^{e}\ar@/_1.5pc/[r]_{f} & v
	}
	\]
  \end{example}
	
	In sharp contrast, sections are far less constrained when the vertices of $Q$ admit at most one incoming edge. 
	
	\begin{example}
	\label{ex:two_arrow}
	Given vector spaces $U,V,W$ along with linear maps $A:V \to U$ and $B:V \to W$, the sections of the quiver representation
	\[
	\xymatrixcolsep{.6in}
	\xymatrix{
		U & V \ar@{->}[l]_-{A} \ar@{->}[r]^-{B} &W.
	}
	\]
	are triples of the form $\gamma = (Ax,x,Bx)$ for $x$ in $V$. 
	\end{example}
	
	More generally, consider the case where $Q$ admits a distinguished vertex {$\rho$} in $V$ called the {\em root} so that for each other vertex {$v \neq \rho$} there is a unique path $p[v]$ in $Q$ from $\rho$ to $v$. Quivers satisfying this unique path property are studied in various contexts and hence have many names --- these include {\em out-trees}, {\em out-branchings}, {\em directed rooted trees}, and (the far more scenic) {\bf arborescences} \cite[Chapter 9]{bangjensen2009digraphs}.

	\begin{proposition}\label{prop:arbsec}
		Let $\bA_\bullet$ be a representation of an arborescence $Q$ with root vertex {$\rho$}. The space of sections $\Gamma(Q;\bA_\bullet)$ is isomorphic to $\bA_\rho$, with every section $\gamma$ uniquely determined by the vector $x = \gamma_\rho$ in $\bA_\rho$, via
		\[
		\gamma_v = \bA_{p[v]}(x),
		\] 
		where $p[v]$ is the unique path in $Q$ from $\rho$ to $v \neq \rho$.
	\end{proposition}

	Over the next three sections, we will describe an algorithm to compute $\Gamma(Q;\bA)$ for any given representation $\bA_\bullet$ of an arbitrary quiver $Q$. 

	\begin{remark} \label{rmk:whyalg} As described in Definition \ref{def:section}, an edge $e: u \to v$ of the quiver $Q$ imposes $\dim \bA_{s(e)}$ linear constraints (which may not be independent) on $\tot(\bA_\bullet)$. The space of sections $\Gamma(Q;\bA_\bullet)$ is the subspace that satisfies all such constraints, the kernel of a matrix of~size
	    \[
	    \left(\sum_{e \in Q} \dim \bA_{t(e)} \right) \times \dim \tot(\bA_\bullet).
	    \]
	    In principle, this kernel may be computed directly via Gaussian elimination. We take an alternative approach, which makes use of the structure of $Q$, for two compelling reasons:
	\begin{enumerate}
	    \item Working in the space $\tot(\bA_\bullet)$ 
	    quickly becomes prohibitive when the number of vertices or edges of the quiver is large. In contrast, our approach computes the space of sections by performing Gaussian elimination on much smaller matrices. For a thorough complexity analysis, see Section \ref{ssec:complexity}.
	    \item Our approach extends the notion of a
	    spanning arborescence of a quiver to the setting of quiver representations, as follows. Our algorithm constructs a new quiver $Q^+$ with new representation $\bA_\bullet^+$. Here $Q^+$ is an arborescence obtained by adjoining a new root vertex $\rho$ to $Q$ and passing to a spanning arborescence of this union, while $\bA^+_\bullet$ is a representation of $Q^+$ with $\bA^+_v \subset \bA_v$ for each non-root vertex $v$. Crucially, the space of sections $\bA^+_\rho \simeq \Gamma(Q^+;\bA_\bullet^+)$ is isomorphic to $\Gamma(Q;\bA_\bullet)$. We hope that our construction of the pair $(Q^+,\bA^+_\bullet)$ will be of independent interest.
	\end{enumerate}
	\end{remark}

	\section{Sections of Strongly Connected Quivers}
	\label{sec:scq} 
	
	A quiver $Q = (s,t:E \to V)$ is called {\bf strongly connected} if for any pair of vertices $v,v'$ in $V$ there is at least one path from $v$ to $v'$. The simplest examples of strongly connected quivers are cycles, but such quivers can be far more intricate. In this section, we study sections of strongly connected quivers. We will use a particular decomposition of such quivers into a union of simpler quivers. To this end, note that a {\bf subquiver} $Q' \subset Q$ is a choice of subsets $V' \subset V$ and $E' \subset E$ so that the restrictions of $s$ and $t$ to $E'$ take values in $V'$. For example, every path $(e_1,\ldots,e_k)$ in $Q$ forms a subquiver with 
	\[
	V' = \set{s(e_1), \, s(e_2), \, \ldots, \, s(e_k), \, t(e_k) } \,  \text{ and } \, E' = \set{e_1, \ldots, e_k },
	\]
    where $[k] = \set{1,\ldots,k}$.
    In the special case where a subquiver $Q'$ comes from a path in $Q$, we define its source and tail $s(Q')$ and $t(Q')$ to be the source and tail of that path. Here is the decomposition of interest \cite[Sec 5.3]{bangjensen2009digraphs}.
	
	\begin{definition}\label{def:eardecomp}
		An {\bf ear decomposition} $Q_\bullet$ of $Q$ is an ordered sequence of $c \geq 1$ subquivers $\set{Q_i=(s_i,t_i:E_i \to V_i) \mid i \in [c]}$ of $Q$ subject to the following axioms:
		\begin{enumerate}
			\item the edge sets $E_i$ partition $E$ --- in other words, they are mutually disjoint and their union equals $E$; moreover,
			\item the quiver $Q_1$ is either a single vertex or a cycle, while $Q_i$ for each $i > 1$ is a (possibly cyclic) path in $Q$; and finally,
			\item for each $i > 1$, the intersection of $V_i$ with the union $\bigcup_{j < i}V_j$ equals $\set{s(Q_i),t(Q_i)}$; this intersection has cardinality $1$ if $Q_i$ is a cycle and cardinality $2$ otherwise.
		\end{enumerate}
	\end{definition}
	
	Ear decompositions play an important role in the study of strongly connected quivers due to the following fundamental result.
	
	\begin{theorem}\label{thm:allears}  A quiver with at least two vertices is strongly connected if and only if it has an ear decomposition.
	\end{theorem}
	
	 \noindent At least one standard proof of this result is given in the form of an efficient algorithm for constructing ear decompositions --- see \cite[Theorem  5.3.2]{bangjensen2009digraphs} for details. The figure below illustrates a strongly connected subquiver of the quiver depicted in the Introduction along with its decomposition into three ears:
	
	\begin{center}
	   \includegraphics[scale=.45]{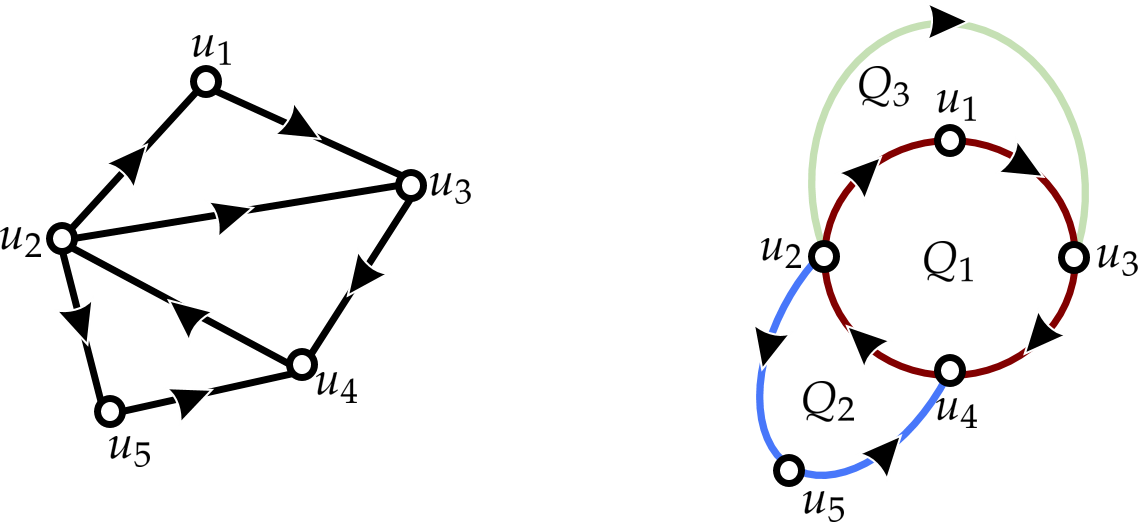}
	\end{center}
	
    We assume for the remainder of this section that $Q$ is strongly connected and fix an ear decomposition $Q_\bullet$ as in Definition \ref{def:eardecomp}. 
    The {\em depth} of an edge $e \in E$, denoted $|e|$, is the unique $i \in [c]$ with $e \in E_i$. We say that a path $p = (e_1,\ldots,e_k)$ is $Q_\bullet$-increasing if the $|e_i|$ form a weakly increasing sequence, and define $Q_\bullet$-decreasing paths analogously. For each vertex $v \in V$, we write $\ell(v)$ for the smallest $i$ in $[c]$ such that $v \in V_i$.
	
	\begin{proposition}\label{prop:monotonepaths}
		Let $\rho$ be any vertex in $V_1$. For any vertex $v \neq \rho$ in $V$, there exists
		\begin{enumerate}
			\item a unique $Q_\bullet$-increasing path $p[v]$ from $\rho$ to $v$ with all edges of depth $\leq \ell(v)$, and
			\item a unique $Q_\bullet$-decreasing path $q[v]$ from $v$ to $\rho$ with all edges of depth $\leq \ell(v)$. 
		\end{enumerate}
	\end{proposition}
	\begin{proof}
		For $\ell(v) = 1$, the desired conclusion follows immediately because $Q_1$ must be a cycle by axiom (2) of Definition \ref{def:eardecomp}. Proceeding inductively, we assume that the assertion holds whenever $\ell(v) < i$, and consider any $v \in V_i$. Once again by axiom (2), our vertex $v$ lies on a path $Q_i$ from $s(Q_i)$ to $t(Q_i)$; and by axiom (3), the inductive hypothesis applies to both $s(Q_i)$ and $t(Q_i)$. The increasing path $p[v]$ is built by first going from $\rho$ to $s(Q_i)$ along $p[s(Q_i)]$ and then onward to $v$ along $Q_i$.  Similarly, the decreasing path $q[v]$ is built by concatenating the piece of $Q_i$ which goes from $v$ to $t(Q_i)$ with the path $q[t(Q_i)]$.
	\end{proof}
	
	For each $i$ in $[c]$, the set $E_i$ contains at most one edge $\epsilon_i \in E_i$ whose target is $t(Q_i)$; we allow for the possibility that $\epsilon_1$ does not exist if $Q_1$ has no edges, but all other $\epsilon_i$ exist and are uniquely determined by the ear decomposition.
	We call $\epsilon_i$ the {\em $i$-th terminal edge} with respect to the ear decomposition $Q_\bullet$, and denote the set of all terminal edges by $E_\text{ter} \subset E$. 
	
	\begin{definition} \label{def:eararb} 
		The {\bf arborescence induced by $Q_\bullet$} is the subquiver $T = T(Q_\bullet)$  with vertex set $V$ and edges $E - E_\text{ter}$.
	\end{definition}

	\noindent To confirm that $T$ is an arborescence, note that its root vertex is $\rho = s(Q_1)$, and that for any other vertex $v$ there is a unique path $p[v]$ from $\rho$ to $v$, whose existence is guaranteed by Proposition \ref{prop:monotonepaths}. In the ear decomposition drawn above,
	the three terminal edges (with respect to the root vertex $u_1$)
	are replaced by dotted arcs in the figure below. The arborescence induced by $Q_\bullet$ is obtained by removing these three edges:
	
	\begin{center}
	    \includegraphics[scale=.45]{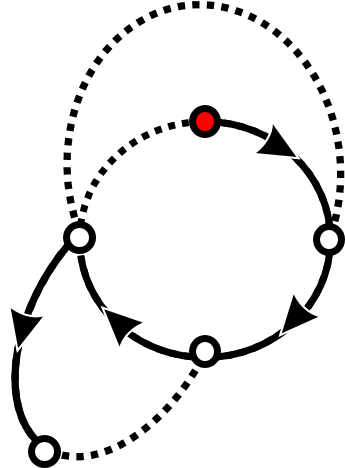}
	\end{center}

	Given a terminal edge $\epsilon \in E_\text{ter}$, consider the linear map $\Delta_\epsilon:\bA_\rho \to \bA_{t(\epsilon)}$ given by
	\begin{align}\label{eq:delta}
		\Delta_\epsilon = \bA_{p[t(\epsilon)]} - \bA_{\epsilon} \circ \bA_{p[s(\epsilon)]}.
	\end{align}
	The kernel of each such map is a subspace $\ker \Delta_\epsilon \subset \bA_\rho$. These kernels depend on the choice of ear decomposition $Q_\bullet$ and the representation $\bA_\bullet$. Let's denote their intersection by
	\begin{align}\label{eq:kerDelta}
		K(Q_\bullet;\bA_\bullet) := \bigcap_\epsilon ~ \ker \Delta_\epsilon,
	\end{align}
	where $\epsilon$ ranges over $E_\text{ter}$.
	This intersection of kernels is independent of the ear decomposition, since it is also the intersection of all $\ker(\bA_p - \bA_q)$, where $p$ and $q$ are any two paths from $\rho$ to the same vertex $v$. 
 We have the following result.
	
	\begin{lemma}\label{lem:scred}
		Let $Q = (s,t:E \to V)$ be a strongly connected quiver with ear decomposition $Q_\bullet$. For any representation $\bA_\bullet$ of $Q$, there is an isomorphism 
		\[
		\Gamma(Q;\bA_\bullet) \simeq K(Q_\bullet;\bA_\bullet)
		\] 
		between the space of sections of $\bA_\bullet$ over $Q$ and the intersection of the kernels from \eqref{eq:kerDelta}.
	\end{lemma}
	\begin{proof}
		Let $T$ be the arborescence induced by $Q_\bullet$ and $\rho$ its root vertex. Using Proposition \ref{prop:arbsec}, vectors in $\bA_\rho$ correspond bijectively with sections in $\Gamma(T;\bA_\bullet)$ via the assignment that sends each $x$ in $\bA_\rho$ to the section given by
		\[
		v \mapsto \gamma_v = \bA_{p[v]}(x).
		\] The subspace $\Gamma(Q;\bA_\bullet) \subset \Gamma(T;\bA_\bullet)$, is obtained by additionally enforcing compatibility across the edges in $E_\text{ter}$. Let $\epsilon$ be a terminal edge and $x_\rho$ a vector in $\bA_\rho$. Now the section $v \mapsto \bA_{p[v]}(x_\rho)$ of $\Gamma(T;\bA_\bullet)$ satisfies the compatibility requirement $\bA_\epsilon(x_{s(\epsilon)}) = x_{t(\epsilon)}$ across $\epsilon$ if and only if $x_\rho$ lies in the kernel of the map $\Delta_\epsilon$ from \eqref{eq:delta}. Thus, our $x_\rho$-induced section is compatible across all the terminal edges if and only if $x_\rho$ lies in $K(Q_\bullet;\bA_\bullet)$.
	\end{proof}
	
	We may safely combine this result with Proposition \ref{prop:arbsec} to reduce a strongly connected quiver to an arborescence while preserving the space of sections.
	
	\begin{corollary} \label{cor:scarb}
		Assuming the hypotheses of Lemma \ref{lem:scred}, let $T$ be the arborescence induced by $Q_\bullet$ and $\rho$ its root vertex. Let $\bA'_\bullet$ be the representation of $T$ prescribed by the following assignments to vertices $v \in V$ and non-terminal edges $e \in E - E_\text{\rm ter}$:
		\[
		\bA'_v := \begin{cases}
		\bA_v & v \neq \rho ,\\
		K(Q_\bullet;\bA_\bullet) & v = \rho ;
		\end{cases} \quad \text{ and } \quad
		\bA'_e := \begin{cases}
		\bA_e & s(e) \neq \rho,\\
		\bA_e\big|_{K(Q_\bullet;\bA_\bullet)} & s(e)=\rho.
		\end{cases}		 					
		\] 
		Then, there is an isomorphism of sections
		\[
		\Gamma(Q;\bA_\bullet) \simeq \Gamma(T;\bA'_\bullet).
		\]
	\end{corollary}
	
	We will use Corollary \ref{cor:scarb} to perform section-preserving simplifications of arbitrary (i.e., not necessarily strongly connected) quivers.
	
	\section{The Acyclic Reduction} 
	\label{sec:acyclic} 
	
	Fix a quiver $Q = (s,t:E \to V)$. A strongly connected subquiver $R \subset Q$ is {\em maximal} if it is not contained in a strictly larger strongly connected subquiver of $Q$. We denote the set of all maximal strongly connected subquivers of $Q$ by $\MSC(Q)$. This set can be extracted from $Q$ very efficiently by employing the remarkable algorithm of Tarjan \cite[Section 5.2]{bangjensen2009digraphs}. Distinct subquivers  in $\MSC(Q)$ have disjoint vertices\footnote{If $R \neq R'$ in $\MSC(Q)$ share a vertex $v$, then there is a path (passing through $v$) in the union $R \cup R'$ from any vertex of $R$ to any vertex of $R'$ and vice versa. The existence of such paths makes $R \cup R'$ strongly connected, contradicting the maximality of either $R$ or $R'$.}. For each $R$ in $\MSC(Q)$, fix an ear decomposition $R_\bullet$ as in Definition \ref{def:eardecomp}. We write $T(R_\bullet)$ for the arborescence induced by $R_\bullet$ as in Definition~\ref{def:eararb}, and let $E_\text{ter}(R_\bullet) \subset E$ be the set of terminal edges of~$R_\bullet$.
	
	\begin{definition}\label{def:acycred}
		The {\bf acyclic reduction} $Q^*$ of $Q =(s,t:E \to V)$ with respect to the ear decompositions $\set{R_\bullet \mid R \in \MSC(Q)}$ is the subquiver $Q^* \subset Q$ defined as follows: it has the same vertex set $V$, while its edge set $E^* \subset E$ is given by removing all terminal edges, i.e.,
		\[
		E^* = E - \bigcup_{R} E_\text{ter}(R_\bullet),
		\]
		where $R$ ranges over $\MSC(Q)$.
	\end{definition}
	
	We note that the quiver $Q^*$ is indeed acyclic (as suggested by its name) as follows. Each cycle in $Q$ is strongly connected, hence lies in a single maximal strongly connected component $R \in \MSC(Q)$. But the removal of all the terminal edges $E_\text{ter}(R_\bullet)$ turns $R$ into the arborescence $T(R_\bullet)$, which cannot contain any cycles. Depicted below is the quiver from the Introduction; the light-shaded edges lie within strongly-connected subquivers, whose root vertices are coloured red. The dotted edges are terminal for the associated ear decompositions, and their removal produces the acyclic reduction:
	
	\begin{center}
	    \includegraphics[scale=.6]{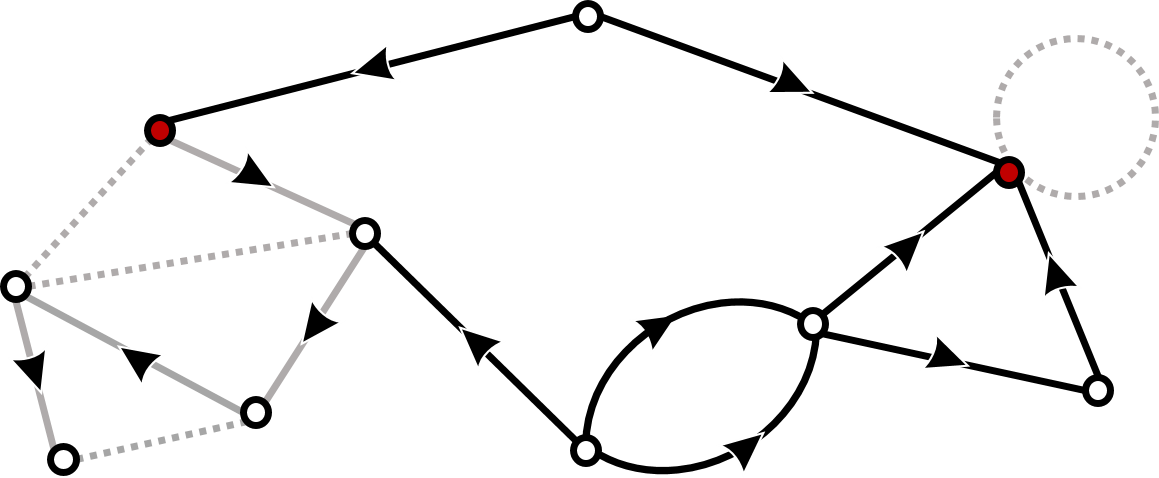}
	\end{center}

	Our next goal is to reduce a given representation $\bA_\bullet$ of $Q$ to a new representation $\bA_\bullet^*$ of $Q^*$ in a manner that preserves the space of sections. Let $\rho: \MSC(Q) \to V$ be the injective {\em root map}, which sends each maximal strongly connected subquiver $R \subset Q$ to the root vertex of $T(R_\bullet)$. We associate to each vertex $v \in V$ the subspace $\bA^\circ_v \subset \bA_v$ given by
	\[
	\bA^\circ_v := \begin{cases} 
	\bA^R_{v} & \text{if } v = \rho(R) \text{ for some } R \in \MSC(Q), \\
	\bA_v & \text{otherwise},			
	\end{cases}
	\]
	where, for each $R \in \MSC(Q)$, we write $\bA^R_\bullet$ for the representation of  $T(R_\bullet)$ described in Corollary \ref{cor:scarb}. 
	
	\begin{definition}\label{def:LambdavR}
		For each vertex $v \in V$ and strongly connected $R \in \MSC(Q)$, let $P^*_{v\to R}$ be the set of all paths in $Q^*$ with source $v$ and target $\rho(R)$. The {\bf $R$-constrained space at $v$} is the subspace $\Lambda_{v,R} \subset \bA^\circ_v$ given by
		\[
		\Lambda_{v,R} := \set{x \in \bA^\circ_v \mid \bA_p(x) \in \bA^R_{\rho(R)} \text{ for all } p \in P^*_{v\to R}},
		\] 
		with the implicit understanding that $\Lambda_{v,R}$ equals $\bA^\circ_v$ whenever $P^*_{v\to R}$ is empty. 
	\end{definition} 
	
	Our next result shows that $R$-constrained subspaces behave well under the linear maps assigned by $\bA_\bullet$ to edges of $Q^*$.
	
	\begin{proposition}\label{prop:lambdarest}
		For any edge $e$ in $E^*$ and subquiver $R \in \MSC(Q)$, the linear map $\bA_e:\bA_{s(e)} \to \bA_{t(e)}$ sends $\Lambda_{s(e),R}$ to $\Lambda_{t(e),R}$.
	\end{proposition}
	\begin{proof}
		Let $p = (e_1,\ldots,e_k)$ be any path in $P^*_{t(e) \to R}$ and note that the augmented path $p' = (e,e_1,\ldots,e_k)$ is an element of $P^*_{s(e) \to R}$. Now for any $x$ in $\Lambda_{s(e),R}$ we know that $\bA_{p'}(x)$ lies in $\bA^R_{\rho(R)}$ by Definition \ref{def:LambdavR}. But $\bA_{p'}(x)$ is $\bA_p \circ \bA_e(x)$, whence $\bA_e(x)$ lies in $\Lambda_{t(e),R}$.
	\end{proof}
	
	\noindent Consider the intersection of all the $R$-constrained spaces at a given vertex $v \in V$, i.e, define the subspace $\Lambda_v \subset \bA_v$ as
	\begin{align}\label{eq:Lambdav}
		\Lambda_v := \bigcap_{R} \Lambda_{v,R}
	\end{align}
	where $R$ ranges over $\MSC(Q)$. It follows immediately from Proposition \ref{prop:lambdarest} that for each edge $e$ in $E^*$ the map $\bA_e$ sends $\Lambda_{s(e)}$ to $\Lambda_{t(e)}$.  
	
	\begin{definition}\label{def:acycrep}
		Let $Q^*$ be the acyclic reduction of $Q$ with respect to a choice of ear decompositions $\set{R_\bullet \mid R \in \MSC(Q)}$. The {\bf acyclification} of a representation $\bA_\bullet$ of $Q$ is a new representation $\bA_\bullet^*$ of $Q^*$ which assigns to every vertex $v$ in $V$ the vector space
		\[
		\bA^*_v = \Lambda_v
		\]
		and to every edge $e$ in $E^*$ the restriction of $\bA_e$ to $\Lambda_{s(e)}$, denoted $\bA^*_e:\Lambda_{s(e)} \to \Lambda_{t(e)}$.
	\end{definition}
	
	As promised, our new representation $\bA^*_\bullet$ retains full knowledge of the sections of the original representation $\bA_\bullet$ even though it is only defined on the acyclic reduction $Q^*$.
	
	\begin{proposition}\label{prop:acycred}
		Let $\bA_\bullet$ be a representation of a quiver $Q = (s,t:E \to V)$. Writing $Q^*$ for the acyclic reduction of $Q$ with respect to some choice of ear decompositions $\set{R_\bullet \mid R \in \MSC(Q)}$ and $\bA_\bullet^*$ for the corresponding acyclification of $\bA_\bullet$, there is an isomorphism of sections
		\[
		\Gamma(Q;\bA_\bullet) \simeq \Gamma(Q^*;\bA^*_\bullet).
		\]	
	\end{proposition}
	\begin{proof}
		First we show that a section $\gamma$ in $\Gamma(Q;\bA_\bullet)$ gives a section in $\Gamma(Q^*;\bA^*_\bullet)$. Since $E^* \subset E$ by Definition \ref{def:acycred}, it suffices to prove that $\gamma_v$ lies in the subspace $\Lambda_v$ of $\bA_v$ for all vertices $v$ in $V$. Since $\gamma$ restricts to a section in $\Gamma(R;\bA_\bullet)$ for every subquiver $R \in \MSC(Q)$, it follows from Corollary \ref{cor:scarb} that $\gamma_{\rho(R)}$ lies in the subspace $\bA^R_{\rho(R)}$ of $\bA_{\rho(R)}$. Thus, for any vertex $v \in V$ and every path $p$ in $P^*_{v\to R}$, compatibility forces $\bA_p(\gamma_v) \in \bA^R_{\rho(R)}$. Thus, $\gamma_v$ must lie in the subspace $\Lambda_v$ from \eqref{eq:Lambdav}. Now consider any edge $e \in E^*$ and note that $\bA^*_e$ is defined simply by restricting $\bA_e$ to the subspace $\Lambda_{s(e)}$. Thus, we obtain 
		\[
		\bA^*_e(\gamma_{s(e)}) = \bA_e(\gamma_{s(e)}) =  \gamma_{t(e)}
		\] for each such edge, and it follows that $\gamma$ is a section in $\Gamma(Q^*;\bA^*_\bullet)$. Conversely, consider a section $\gamma^*$ in $\Gamma(Q^*;\bA^*_\bullet)$. The $\bA_\bullet$-compatibility of $\gamma^*$ across every edge $e \in E^*$ follows from the fact that $\bA^*_e$ is the restriction of $\bA_e$; it therefore suffices to show that $\gamma^*$ is also $\bA_\bullet$-compatible across all the edges in $E - E^*$. By Definition \ref{def:acycred}, any such edge $\epsilon$ lies in $E_\text{ter}(R_\bullet)$ for a unique $R \in \MSC(Q)$. We know that $\Lambda_{\rho(R)}$ is a subspace of $\bA^R_{\rho(R)}$, 
		by \eqref{eq:Lambdav} combined with Definition \ref{def:LambdavR}. Thus, Corollary \ref{cor:scarb} guarantees that $\gamma^*$ is also $\bA_\bullet$-compatible across $\epsilon$, as desired. 
	\end{proof}
	
	\section{The Arboreal Replacement}
    \label{sec:arboreal}

	We assume here that $Q = (s,t:E \to V)$ is an acyclic quiver, so its vertex set $V$ is partially ordered by (the reflexive closure of) the binary relation
	\[
		u < v \text{ if and only if there is a path }p \text{ in } Q \text{ with } s(p) = u \text{ and } t(p) = v.
	\]
	Let $V_\text{min} \subset V$ be set of all minimal vertices with respect to this partial order --- thus, a vertex $v$ lies in $V_\text{min}$ if and only if there is no edge $e \in E$ with $t(e) = v$. We fix a representation $\bA_\bullet$ of $Q$, and seek to  compute the space of sections $\Gamma(Q;\bA_\bullet)$. For this purpose, it will be convenient to formally add a new vertex to $Q$ that serves as the global minimum for the partial order described above.
	
	\begin{definition}\label{def:augquiv}
	The {\bf augmented quiver} $Q^+$ has vertices $V^+ := V \cup \set{\rho}$, where $\rho$ is a new vertex. Its edge set $E^+$ is $E \cup \set{e_v \mid v \in V_\text{min}}$; the sources and targets of edges in $E$ are inherited from $Q$, while each new edge $e_v$ has source $\rho$ and target $v$ in $V_\text{min}$.
	\end{definition}
	
	Drawn below is the augmented quiver corresponding to the acyclic reduction from the previous section; the root and (two) new edges $e_v$ for the vertices $v \in V_\text{min}$ are highlighted in blue.
	
	\begin{center}
	    \includegraphics[scale=.6]{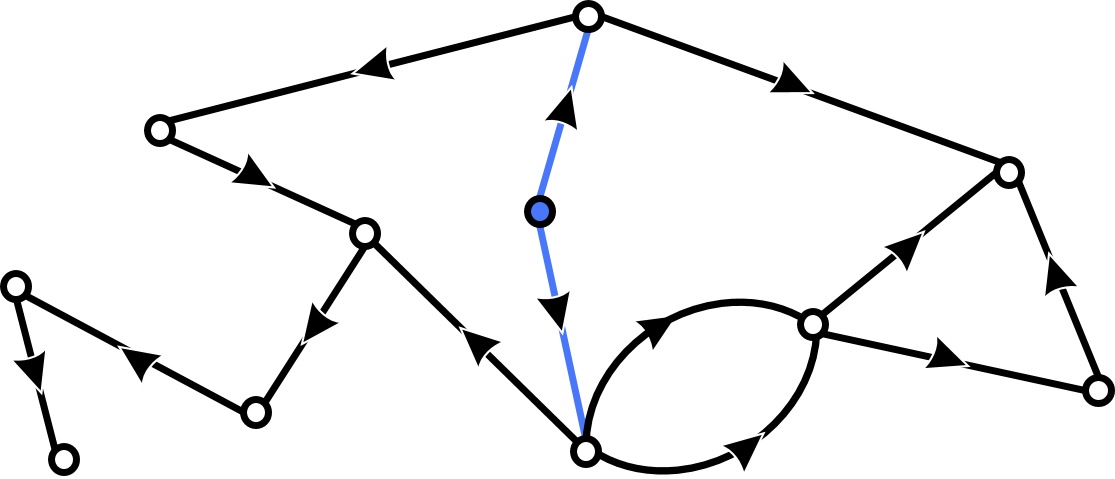}
	\end{center}
	
	\noindent A representation $\bA_\bullet$ extends to $Q^+$ if we define 
	\[
	\bA_\rho := \hspace{-2mm} \prod_{v \in V_\text{min}} \hspace{-2mm} \bA_v,
	\] and let $\bA_{e_v}:\bA_\rho \to \bA_v$ be the canonical projection map. Now each section of $\bA_\bullet$ over $Q$ extends uniquely to a section over $Q^+$, whence we have an isomorphism
	\begin{align}\label{eq:Qplussec}
	\Gamma(Q;\bA_\bullet) \simeq \Gamma(Q^+;\bA_\bullet).
	\end{align}
	Thus, there is no loss of generality encountered when computing the sections of $\bA_\bullet$ over $Q^+$ rather than $Q$. We will also make frequent use of the following notion.
	
	\begin{definition}\label{def:eq}
		Let $n \geq 1$ be a natural number and $X,Y$ a pair of vector spaces. The {\bf equaliser} of a collection of $n$ linear maps $\set{f_i:X \to Y \mid 1 \leq i \leq n}$ is the largest subspace $\Eq\set{f_\bullet} \subset X$ satisfying $f_i(x) = f_j(x)$ for all $x$ in $\Eq\set{f_\bullet}$ and all $i,j$ in $\set{1,\ldots,n}$.
	\end{definition}

	 In practice,  for finite-dimensional $X$ the equaliser $\Eq\set{f_\bullet}$  can be computed by intersecting kernels of successive differences:
	\[
	\Eq\set{f_\bullet} = \bigcap_{i=1}^{n-1} \ker\left(f_i - f_{i+1}\right),
	\]
	with the understanding that for $n = 1$ this intersection over the empty set equals all of $X$. 
	
	\begin{definition}\label{def:flow}
		Assign to each vertex $v \in V^+$ a subspace $\Phi_v \subset \bA_\rho$ and a linear map $\phi_v: \Phi_v \to \bA_v$, called the {\bf flow space} and {\bf flow map} of $\bA_\bullet$ at $v$, inductively over the partial order $\leq$ as follows:
		\begin{enumerate}
			\item for $v = \rho$, the flow space $\Phi_\rho$ equals $\bA_\rho$, and the flow map $\phi_\rho:\Phi_\rho \to \bA_\rho$ is the identity; 
			\item for $v \neq \rho$, let $E_\text{in}(v) \subset E^+$ be the (necessarily nonempty) set of all edges $e$ satisfying $t(e) = v$. Noting that $s(e) < v$ for any such $e$, define the subspace $\Phi'_v \subset \bA_\rho$ via
			\[
			\Phi'_v := \bigcap_{e} \Phi_{s(e)},
			\]
			where $e$ ranges over $E_\text{in}(v)$. For each such $e$, the composition $\bA_{e} \circ \phi_{s(e)}$ restricts to a linear map $\Phi'_v \to \bA_v$. The flow space at $v$ is the equaliser
			\[
			\Phi_v := \Eq\set{\bA_{e} \circ \phi_{s(e)}:\Phi'_v \to \bA_v \mid e \in E_\text{in}(v)}.
			\]
			The flow map $\phi_v:\Phi_v \to \bA_v$ is given by $\bA_{e} \circ \phi_{s(e)}$ for any $e$ in $E_\text{in}(v)$.
		\end{enumerate}
	\end{definition}
		
	By construction, the flow space $\Phi_v$ for a vertex $v \neq \rho$ forms a subspace of the intersection $\bigcap_{u}\Phi_{u}$ of flow spaces ranging over all preceding vertices $u < v$. Thus, we can restrict the flow map $\phi_{u}$ at $u$ to a vector in the flow space $\Phi_v$ whenever $u \leq v$. Our affinity for flow spaces and maps stems mainly from the following result.
	
	\begin{proposition}\label{prop:flowsec}
	For each vertex $v \in V^+$, let $Q^+_{\leq v}$ be the subquiver of $Q^+$ generated by all vertices $u \leq v$ and the edges between them. Then, $\gamma$ is a section in $\Gamma(Q^+_{\leq v};\bA_\bullet)$ if and only if the vector $\gamma_\rho \in \bA_\rho$ lies in the flow space $\Phi_v$.
	\end{proposition}
	\begin{proof}
	For $v = \rho$ the result holds because in this case the spaces below are all equal:
	\[
	\Phi_\rho = \Gamma(Q^+_{\leq \rho};\bA_\bullet) = \bA_\rho,
	\]
	 with the flow map $\phi_\rho:\Phi_\rho \to \bA_\rho$ being the identity. Proceeding inductively over the partial order $\leq$, consider any $v \neq \rho$ and assume that the desired result holds for all preceding vertices $u < v$. We must show that any $x \in \Phi_v$ generates a section in $\Gamma(Q^+_{\leq v};\bA_\bullet)$ via the assignment $u \mapsto \phi_u(x)$ for every $u \leq v$. Compatibility for all edges $e$ with $t(e) \neq v$ follows from the inductive hypothesis, so it suffices to examine all edges $e \in E_\text{in}(v)$. For any such edge, Definition \ref{def:flow} yields
	 \[
	 \bA_e \circ \phi_{s(e)}(x) = \phi_v(x),
	 \]
	 hence establishing the desired compatibility. 	 Conversely, if $\gamma$ is a section in $\Gamma(Q^+_{\leq v};\bA_\bullet)$ then it suffices to show that the vector $\gamma_\rho \in \bA_\rho$ lies in the subspace $\Phi_v$. By the inductive hypothesis, we have
	 \[
	 \gamma_\rho \in \Phi'_v = \bigcap_e \Phi_{s(e)},
	 \]
	 where $e$ ranges over the edges in $E_\text{in}(v)$. By compatibility of $\gamma$ across any such $e$, we have
	 \[
	 \bA_e\circ\phi_{s(e)}(\gamma_{\rho}) = \gamma_v,
	 \]
	 so $\gamma_\rho$ lies in the equaliser $\Phi_v = \Eq\set{\bA_e \circ \phi_{s(e)} \mid e \in E_\text{in}(v)}$ as desired.
	\end{proof}
	
	Using the fact that the quiver $Q^+$ is the union of the subquivers $\set{Q^+_{\leq v} \mid v \in V^+}$, we are able to describe the sections of $\bA_\bullet$ as intersections of its flow spaces. We write $V_\text{\rm max} \subset V$ for the $\leq$-maximal vertices (i.e., the vertices which do not serve as sources of edges in $E^+$).
	
	\begin{proposition} \label{prop:acycsec}
	 Let $\bA_\bullet$ be a representation of an acyclic quiver $Q$, and $Q^+$ the augmented quiver (as in Definition \ref{def:augquiv}). We have an isomorphism
	 \[
	 \Gamma(Q;\bA_\bullet) \simeq \hspace{-2mm} \bigcap_{v \in V_\text{\rm max}} \hspace{-2mm} \Phi_v
	 \]
	 between the sections of $\bA_\bullet$ over $Q$ and the intersection of the flow spaces of $\bA_\bullet$ at the maximal vertices.
	\end{proposition}
	\begin{proof}
	Combining \eqref{eq:Qplussec} with Proposition \ref{prop:flowsec} and the fact that $Q^+ = \bigcup_{v \in V} Q^+_{\leq v}$ gives
	\[
	\Gamma(Q;\bA_\bullet) \simeq \bigcap_{v \in V} \Phi_v.
	\]
     Since maximal vertices have the smallest flow spaces, by Definition \ref{def:flow}, the desired result follows. 
	\end{proof}
	
	For brevity, we write $\Phi(\bA_\bullet)$ to indicate the intersection $\bigcap_v \Phi_v$ of flow spaces ranging over $V_\text{max}$ (or, equivalently, over $V$). By employing breadth-first search \cite[Chapter 3.3]{bangjensen2009digraphs} on $Q^+$ starting at $\rho$, one can construct a spanning arborescence $T^+ \subset Q^+$ with root $\rho$. This arborescence $T^+$ must necessarily contain all the vertices in $V^+$ and all the edges in  $(E^+ - E)$, but in general it is not uniquely determined otherwise. One possible spanning arborescence for the augmented quiver drawn above is obtained by removing the light-shaded edges below:
	
    \begin{center}
    \includegraphics[scale=.6]{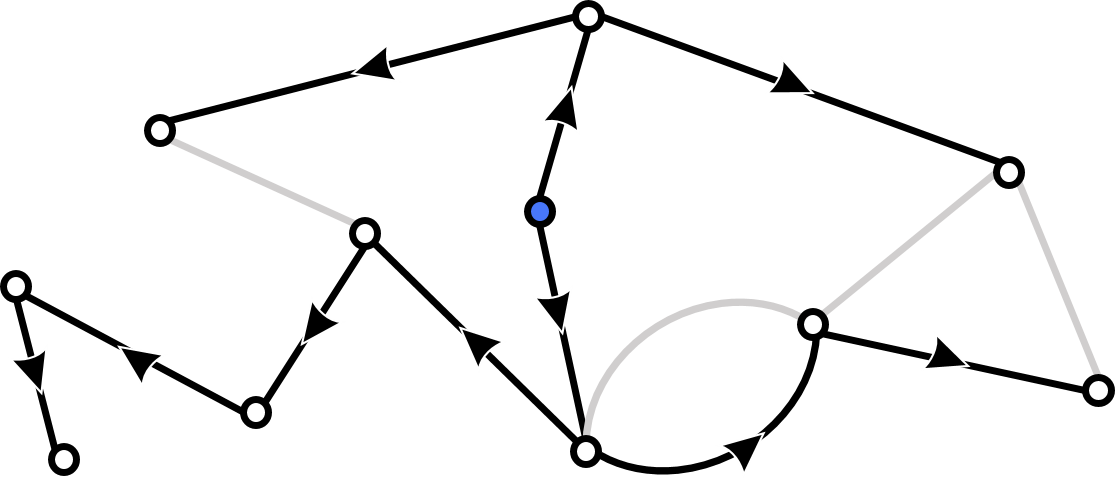}
    \end{center}
	
	\begin{definition}\label{def:arbrepl}
	Let $T^+ \subset Q^+$ be any spanning arborescence with root $\rho$. An {\bf arboreal replacement} of $\bA_\bullet$ is the representation $\bA^+_\bullet$ of $T^+$ that assigns
	\[
		\bA^+_v := \begin{cases}
		\bA_v & v \neq \rho,\\
		\Phi(\bA_\bullet) & v = \rho;
		\end{cases} \quad \text{ and } \quad
		\bA^+_e := \begin{cases}
		\bA_e & s(e) \neq \rho,\\
		\bA_e\big|_{\Phi(\bA_\bullet)} & s(e)=\rho.
		\end{cases}		 					
	\]
	\end{definition}
	
	The following result is obtained by combining Proposition \ref{prop:acycsec} with Proposition \ref{prop:arbsec}.
	
	\begin{corollary}\label{cor:arbrep}
	  Let $\bA_\bullet$ be a representation of an acyclic quiver $Q$ and $T^+$ a spanning arborescence of the augmented quiver $Q^+$. There is an isomorphism
	  \[
	  \Gamma(Q;\bA_\bullet) \simeq \Gamma(T^+;\bA^+_\bullet)
	  \]
	  between the sections of $\bA_\bullet$ and those of its arboreal replacement $\bA_\bullet^+$ defined on $T^+$.
	\end{corollary}
	
	\section{The Space of Sections}
	\label{sec:dimension}

 We are now ready to establish Theorem (A) from the Introduction. 
 
\begin{theorem}\label{thm:mainrev}
For any representation $\bA_\bullet$ of a quiver $Q = (s,t:E \to V)$, the following spaces are all isomorphic: 
\begin{align}\label{eq:allsec}
\Gamma(Q;\bA_\bullet) \simeq \Gamma(Q^*;\bA^*_\bullet) \simeq \Gamma(T^+;\bA^+_\bullet) \simeq \bA^+_\rho.
\end{align}
Here, $Q^*$ is the acyclic reduction of $Q$ with $\bA_\bullet^*$ the acyclification of $\bA_\bullet$. Similarly, writing $Q^+$ for the augmented quiver associated to $Q^*$ with root $\rho$, the representation $\bA_\bullet^+$ is the arboreal replacement of~$\bA_\bullet^*$ defined on any spanning arborescence $T^+ \subset Q^+$. 
\end{theorem}	
\begin{proof}
The first isomorphism follows from Proposition \ref{prop:acycred}, the second from Corollary \ref{cor:arbrep}, and the third from Proposition \ref{prop:arbsec}.
\end{proof}

Theorem (A) asserts the existence of an isomorphism $\bA^+_\rho \simeq \Gamma(Q;\bA_\bullet)$ as a map $F$, which we now describe. Assuming the hypotheses and notation of Theorem \ref{thm:mainrev}, there are containments
\[
\bA^+_v \subset \bA^*_v \subset \bA_v,
\]
for each vertex $v$ in $V$, by Definition \ref{def:acycrep} and Definition \ref{def:arbrepl}. 
Since $T^+$ is an arborescence, it admits a unique path $p[v]$ from its root $\rho$ to any such $v$. This path carries a linear map $\bA^+_{p[v]}:\bA^+_\rho \to \bA^+_v$, and the collection of all such linear maps (indexed over $v \in V$) assembles to furnish a single map to the direct product:
\[
\bA^+_\rho \to \prod_v \bA^+_v.
\]
In light of the containments $\bA^+_v \subset \bA_v$ described above, the codomain is a subspace of $\tot(\bA_\bullet)$. Thus, we obtain a linear map
\begin{align} \label{eq:F}
F:\bA^+_\rho \to \tot(\bA_\bullet),
\end{align}
whose image of $F$ inside $\tot(\bA_\bullet)$ is an isomorphically embedded copy of $\Gamma(Q;\bA_\bullet)$. Although various choices (of ear decompositions and spanning arborescences) made above will produce different $F$'s, the image of $F$ remains invariant. 

\subsection{Lower bounds on the dimension}

As stated in the Introduction, we will define principal components along $\bA_\bullet$ as solutions to an optimisation problem over $\Gamma(Q;\bA_\bullet)$. In order for this to be a non-trivial problem, one requires the dimension $d := \dim \Gamma(Q;\bA_\bullet)$ to exceed zero.
We therefore take a brief detour here in order to highlight some sufficient conditions (on $Q$ and $\bA_\bullet)$ which give lower bounds on $d$.  Among the simplest cases to analyse in terms of the topology of $Q$ are the extreme ones, as recorded in the following observation.

\begin{proposition}
\label{prop:two-observations}
Let $\bA_\bullet$ be a representation of a quiver $Q$.
\begin{enumerate}
    \item if $Q$ is an arborescence with root $\rho$, then $d = \dim \bA_\rho$; and
    \item if $Q$ is strongly connected, then $d = 0$ for all sufficiently generic $\bA_\bullet$.
\end{enumerate}
\end{proposition}
\begin{proof}
The first assertion follows directly from Proposition \ref{prop:arbsec}, so we concentrate on the second assertion. Let $Q_\bullet$ be an ear decomposition of $Q$ (see Definition \ref{def:eardecomp}) and $\rho$ a vertex in $Q_1$. By strong connectedness, there exists a path $p$ in $Q$ from $\rho$ to itself, which carries an endomorphism $\bA_p:\bA_\rho \to \bA_\rho$. Now any section $\gamma$ of $Q$ must satisfy $\bA_p(\gamma_\rho) = \gamma_\rho$. For generic $\bA_\bullet$, this endomorphism $\bA_p$ will not have $1$ as an eigenvalue, so $\gamma_\rho$ must be zero. The result now follows from applying Proposition \ref{prop:arbsec} to the arborescence induced by $Q_\bullet$.
\end{proof}

Although the result in part (2) of Proposition~\ref{prop:two-observations} might appear disappointing at first glance, we note that there are several interesting non-generic families of linear maps which do admit $1$ as an eigenvalue, such as those arising from row-stochastic matrices. Moreover, general quivers are neither strongly connected nor arboreal but lie somewhere in between. Using Proposition \ref{prop:acycred}, any representation of an arbitrary quiver can be reduced to a representation of an acyclic quiver while preserving $d$, so it remains to provide lower bounds on $d$ for representations of acyclic quivers.

\begin{proposition}
\label{prop:megatarget}
Let $Q$ be an acyclic quiver with minimal vertices $V_\text{\rm min}$ and maximal vertices $V_\text{\rm max}$. For any representation $\bA_\bullet$ of $Q$, we have 
\[
\dim \Gamma(Q; \bA_\bullet) \geq \hspace{-2mm} \sum_{u \in V_\text{\rm min}} \hspace{-2mm} \dim \bA_u - \hspace{-2mm} \sum_{v \in V_\text{\rm max}} \hspace{-2mm} (n_v-1) \dim \bA_v,
\]
where $n_v$ is the total number of paths in the augmented quiver $Q^+$ from the root $\rho$ to the vertex $v$.
\end{proposition}

\begin{proof}
By Definition \ref{def:flow}, the flow space $\Phi_\rho = \bA_\rho$ has dimension $\sum_{u \in V_\text{min}} \dim \bA_u$. We claim that the flow space $\Phi_v$ at a vertex $v \in V_\text{max}$ has codimension at most $(n_v - 1) \dim \bA_v$ in $\Phi_\rho$. To establish this claim, let $\set{f_k:\bA_\rho \to \bA_v \mid 1 \leq k \leq n_v}$ be the linear maps carried by paths from $\rho$ to $v$, and examine the $(n_v-1)$ kernels of the differences $\Delta_k = (f_k - f_{k+1})
$. Since each $\ker(\Delta_k)$ has codimension at most $\dim \bA_v$ in $\Phi_\rho$, and since the codimension of their intersection is at most the sum of these codimensions, we have $\text{codim } \Phi_v \leq (n_v - 1) \dim \bA_v$ as claimed. The inequality in the statement now follows from Proposition \ref{prop:acycsec}. 
\end{proof}

\begin{remark}\label{rmk:zerosec}
The space of sections might be trivial for several interesting representations of acyclic quivers,
Proposition \ref{prop:megatarget} notwithstanding. For example, this occurs frequently in {\em two parameter persistence modules} \cite{carlsson2009multidimensional} which arise from homology groups of bifiltered simplicial complexes. Such a module is a representation $\bA_\bullet$ of the {\em grid} quiver $Q$ whose vertices are identified with integer points $(i,j)$ with $1 \leq i,j \leq \ell$ for some integer $\ell > 0$; there are two edges from each $(i,j)$, one to $(i+1,j)$ and another to $(i,j+1)$.

Since homology is functorial, each square of the form 
\[
\xymatrixcolsep{1in}
\xymatrixrowsep{.62in}
\xymatrix{
\bA_{(i,j+1)} \ar@{->}[r] & \bA_{(i+1,j+1)} \\
\bA_{(i,j)} \ar@{->}[r] \ar@{->}[u]& \bA_{(i+1,j)} \ar@{->}[u]
}\] commutes. It follows that the space of sections of such a quiver representation is isomorphic to $\bA_{(1,1)}$, which might be trivial even though the other $\bA_{(i,j)}$ and the linear maps between them contain relevant information. As a partial remedy, one can fix a vertex $(i_0,j_0)$ of interest and restrict to the largest subquiver $Q_{\geq (i_0,j_0)} \subset Q$ containing all vertices $(i,j)$ with $i_0 \geq i$ and $j_0 \geq j$. This allows us to extract features from representations of $Q$ (as in Section \ref{sec:PCs}) even when the space of sections is trivial.
\end{remark}

\subsection{Algorithms} \label{ssec:algos}
We describe algorithms to compute the space of sections by combining graph theoretic operations on the quiver with linear algebraic operations on the representation.
That is, we give algorithms arising from Corollary~\ref{cor:scarb}, Proposition~\ref{prop:acycsec},
and Corollary~\ref{cor:arbrep}. Quiver representations $\bA_\bullet$ may be stored on computers as directed graphs whose vertices $v$ have non-negative integer weights $\dim \bA_v$ and whose edges $e$ have matrix-valued weights~$\bA_e$.

The first subroutine implements the constructions from Section \ref{sec:scq}: it ear-decomposes a given strongly-connected quiver, produces an arborescence by removing all terminal edges, and updates the overlaid representation $\bA_\bullet$ at the root vertex in accordance with Corollary~\ref{cor:scarb}. We recall that an efficient algorithm for performing ear decomposition may be found in \cite[Section 5.3]{bangjensen2009digraphs}. 

\medskip

\begin{algorithm}[H]
		\SetAlgoLined
		\KwInput{Strongly Connected quiver $R$ with representation $\bA_\bullet$}
		\KwOutput{Arborescence $T(R) \subset R$ with $\bA_\bullet$ modified}  
		
		Set $\set{R_1,\ldots,R_\ell} := {\bf EarDecompose}(R)$ \\
		Set $\rho := $ root of $R_1$ and $K := \bA_\rho$ \\
		\For{$i$ in $(1,\ldots,\ell)$}{
		     Set $\epsilon := $ terminal edge of $R_i$ \\
		     Remove $\epsilon$ from $R_i$ \\
		     Compute $\Delta_\epsilon := \bA_{p[t(\epsilon)]} - \bA_{\epsilon} \circ \bA_{p[s(\epsilon)]}$ \\
		     Set $K = K \cap \ker(\Delta_\epsilon)$
		}
		Set $\bA_\rho := K$ \\
		Return $(R,\bA_\bullet)$  
		\caption{\textbf{SCReduce} }
	\end{algorithm} 
	
\medskip

The second subroutine implements the constructions from Section \ref{sec:acyclic} by computing the acyclification of a given representation $\bA_\bullet$ of a quiver $Q$. It uses Tarjan's efficient algorithm for computing the set $\MSC(Q)$ of maximal strongly connected components \cite[Section 5.2]{bangjensen2009digraphs}. The {\bf BFSEqualise} function invoked in line 5 is an enhancement of the standard breadth-first search algorithm to do the following computation. Starting from the root $\rho$ of a given $R \in \MSC(Q)$, it finds all edges $e$ with target $\rho$, and replaces each vector space $\bA_{s(e)}$ with the $R$-constrained subspace $\Lambda_{s(e),R}$ from Definition \ref{def:LambdavR}. It then recursively repeats this operation, starting from $s(e)$ rather than $\rho$, until all vertices that admit paths to $\rho$ have been processed.

\medskip

    \begin{algorithm}[H]
     \SetAlgoLined
     \KwInput{A representation $\bA_\bullet$ of a quiver $Q$}
     \KwOutput{The acyclification $\bA^*_\bullet$ and the acyclic reduction $Q^*$}
     \caption{\textbf{AcycReduce}}
     Compute $\MSC(Q)$ \\
     \For{$R$ in $\MSC(Q)$}{
        Set $(T(R),\bA^\circ_\bullet) := \text{\bf SCReduce}(R,\bA_\bullet)$ \\
        Set $\rho := $ root of $T(R)$ \\
        {\bf BFSEqualise}$(\rho,Q,\bA_\bullet)$
        }
     Return $(Q,\bA_\bullet)$ 
    \end{algorithm}
    
\medskip

Our final subroutine is based on Section \ref{sec:arboreal}. It takes as input a representation of an acyclic quiver (such as one produced by {\bf AcycReduce}). The algorithm augments the quiver with a new root, and inductively builds the arboreal replacement by constructing flow spaces and maps (see Definition \ref{def:flow}). The subroutine {\bf Augment} builds $Q^+$ from $Q$ (as in Definition \ref{def:augquiv}) and extends the representation $\bA_\bullet$ to $\bA^+_\bullet$ by letting $\bA^+_\rho$ be the product of $\bA_v$ over initial vertices $v$. The function {\bf TopSort} builds a linear ordering of the vertices that respects the path-induced partial order (this is often called a {\em topological sorting} in the graph theory literature). Finally, the function {\bf SpanArb} uses breadth-first search to construct a spanning arborescence $T^+ \subset Q^+$ with root $\rho$.

\medskip

\begin{algorithm}[H]
     \SetAlgoLined
     \KwInput{An acyclic quiver $Q$ with representation $\bA_\bullet$}
     \KwOutput{A spanning arborescence $T^+$ and arboreal replacement $\bA^+_\bullet$}
     \caption{\textbf{ArbReplace}}
     Set $(Q^+,\bA^+_\bullet) := \textbf{Augment}(Q,\bA_\bullet)$ \\
     \textbf{TopSort}$(Q^+)$, label vertices $\set{\rho,v_1,\ldots,v_m}$  \\
     Set $\Phi_\rho := \bA_\rho$ and $\phi_\rho:\Phi_\rho \to \bA_\rho$ the identity map \\
     \For{$i$ in $(1,2,\ldots,m)$}{
           Set $\Phi'_{v_i} = \bigcap_{t(e)=v_i} \Phi_{s(e)}$ \\
           Set $\Phi_{v_i} := \text{Eq}\set{\bA_e \circ \phi_{s(e)}:\Phi'_{v_i} \to \bA_{v_i} \mid t(e)=v_i}$ \\
           Set $\phi_{v_i} := \bA_e \circ \phi_{s(e)}$ for any $e$ with $t(e)=v_i$. \\
           Set $\bA_\rho := \bA_\rho \cap \Phi_{v_i}$ \\
        }
     Set $T^+ := \textbf{SpanArb}(Q^+,\rho)$ \\
     Return $(T^+,\bA_\bullet)$ 
\end{algorithm}

\medskip

To compute the space of global sections $\Gamma(Q;\bA_\bullet)$, we invoke
\begin{align} \label{eq:secall}
\textbf{ArbReplace}\big((\textbf{AcycReduce}(Q,\bA)\big) 
\end{align}
This produces a representation $\bA^+$ of an arborescence $T^+$ so, by Proposition \ref{prop:arbsec}, the vector space $\bA^+_\rho$ at the root vertex $\rho$ yields the space of sections $\Gamma(Q;\bA_\bullet)$. At each non-root vertex $v$ of $T^+$, the output vector space $\bA^+_v$ is a subspace of the original $\bA_v$. Thus, we can compute an embedding $\Gamma(Q;\bA_\bullet) \inj \tot(\bA_\bullet)$: the component $\Gamma(Q;\bA_\bullet) \inj \bA_v$ for vertex $v$ equals $\bA^+_{p[v]}$, where $p$ is the unique path in $T^+$ from $\rho$ to $v$.

\subsection{Computational complexity} \label{ssec:complexity}

Let $Q = (s,t:E \to V)$ be a quiver with $n_V$ vertices and $n_E$ edges. Fix a representation $\bA_\bullet$ of $Q$ with $n_\bA := \max_v\set{\dim \bA_v}$. We assume throughout that scalar operations in the underlying field take $O(1)$ time.

\begin{remark}
Fix a basis for each vector space $\bA_v$, so the linear maps $\bA_e:\bA_{s(e)} \to \bA_{t(e)}$ can be expressed as matrices. Ordering the vertices and edges of $Q$ arbitrarily, let $M = M(\bA_\bullet)$ be the block matrix whose column blocks are indexed by vertices $v \in V$, row blocks are indexed by edges $e \in E$, and whose $(e,v)$-block is 
\[
M_{e,v} := \begin{cases} -\bA_e & \text{if } v = s(e) \\
                        \text{Id}_{\bA_v} & \text{if } v = t(e) \\
                        0 & \text{otherwise.}
\end{cases}
\]
The subspace $\Gamma(Q;\bA_\bullet)$ is the kernel of $M$, cf. Remark \ref{rmk:whyalg}. 
Thus, computing a basis for this space na\"ively requires Gaussian elimination on the augmented matrix 
\[
M' := \left[\text{Id}_{\tot(\bA_\bullet)} ~\mid~ M\T\right].
\] 
In the worst case, $M'$ has $n_Vn_\bA$ rows and $2n_En_\bA$ columns. Since we may need up to $O(n_V^2n_\bA^2)$ row operations, and since each such operation incurs a cost of $O(n_En_\bA)$, the time complexity is $O(n_V^2n_En_\bA^3)$.
\end{remark}
Here we establish the following.

\begin{corollary}
\label{cor:overall_complexity}
The algorithms from Section \ref{ssec:algos} invoked using \eqref{eq:secall} extract a basis for $\Gamma(Q;\bA_\bullet)$ in time
\begin{align*}\label{eq:totalcost}
O\left(n_V(n_V+n_E)n_\bA^3\right).
\end{align*}
\end{corollary}

We prove Corollary~\ref{cor:overall_complexity} in three parts. The first step gives the complexity of {\bf SCReduce}.

\begin{lemma}\label{lem:scredcomp}
If $Q$ is strongly connected and admits an ear decomposition with $\ell$ ears, then the subroutine \textbf{SCReduce} has time complexity $O(n_V+n_E+n_\bA^3\ell)$ when called with input $(Q,\bA_\bullet)$.
\end{lemma}
\begin{proof}
As described in \cite[Exercise 5.18]{bangjensen2009digraphs}, the cost of building an ear decomposition of $Q$ is $O(n_V+n_E)$. The {\bf for} loop spanning lines 3-8 runs $\ell$ times and, in each iteration of this loop, the computational cost is dominated by the kernel intersection in line 7. Computing this intersection requires Gaussian elimination on a matrix of size at most $n_\bA \times 2n_\bA$, which costs $O(n_\bA^3)$. Hence we obtain the complexity bound $O(n_V+n_E+n_\bA^3\ell)$.
\end{proof}

We now estimate the complexity of calling {\bf AcycReduce} with input $(Q,\bA_\bullet)$.

\begin{lemma}\label{lem:acycredcomp}
The computational complexity of {\bf AcycReduce}$(Q,\bA_\bullet)$ is
\[
O\left((n_V^2+n_E)n_\bA^3\right).
\]
\end{lemma}
\begin{proof}
The set of maximal strongly-connected subquivers $\MSC(Q)$ can be computed in time $O(n_V+n_E)$ \cite[Section 5.2]{bangjensen2009digraphs}. Enumerating its elements as $\set{Q_1, \ldots, Q_s}$, for $s \geq 0$, we note that the {\bf for} loop spanning lines 2-5 runs $s$ times. For each $j$ in $\set{1,\ldots,s}$, we let $n_{V,j}$ and  $n_{E,j}$ denote the number of vertices and edges of $Q_j$, and let $\ell_j$ be the number of ears produced when $Q_j$ is ear-decomposed. We know from Proposition \ref{lem:scredcomp} that the call to {\bf SCReduce} in line 3 incurs a cost of $O (n_{V,j}+n_{E,j}+\ell_jn_\bA^3)$.
The call to {\bf BFSEqualise} in line~5 has a worst-case burden $O (n_En_\bA^3)$,
since we must traverse every edge of $Q$ and perform Gaussian elimination on an augmented  $n_\bA \times 2n_\bA$ matrix to compute the restricted subspace (as in Definition \ref{def:LambdavR}) at its source vertex. Thus, the $j$-th iteration of the {\bf for} loop costs
\[
O\left(n_{V,j}+n_{E,j}+(n_E+\ell_j)n_\bA^3\right).
\]
Since the subquivers $Q_j$ are mutually disjoint, we sum the above expression over $j$ in $\set{1,\ldots,s}$ to obtain the total cost incurred by the {\bf for} loop
\[
O\left(n_V+n_E + \left(sn_E + \sum_{j=1}^s \ell_j\right)n_\bA^3\right).
\]
The conclusion follows by discarding small terms and using the fact that $s$ and $\sum_j\ell_j$ are bounded from above by $n_V$ and $n_E$ respectively.
\end{proof}

In practice, the runtime of {\bf AcycReduce} can be improved in the presence of parallel processing, since {\bf SCReduce} may be called on the strongly connected subquivers of $Q$ concurrently. It remains to estimate the complexity of invoking {\bf ArbReplace} on the output $(Q^*,\bA^*_\bullet)$ of {\bf AcycReduce}$(Q;\bA_\bullet)$. We know from Definition \ref{def:acycred} that the vertex set of $Q^*$ coincides with~$V$, though the edge set $E^*$ may be strictly contained in $E$. Moreover, we have $\dim \bA^*_v \leq n_\bA$ for each vertex $v$.

\begin{lemma}\label{lem:arbrepcomp}
The computational complexity of {\bf ArbReplace}$(Q^*,\bA^*_\bullet)$ is
$O\left(n_Vn_En_\bA^3\right)$.
\end{lemma}
\begin{proof}
Augmentation, topological sorting, and the construction of a spanning arborescence (from lines 1, 2 and 10 respectively) are all $O(n_V+n_E)$ operations, so we restrict our focus to the {\bf for} loop spanning lines 4-9. For each integer $j \geq 0$, let $V_j \subset V$ be the (possibly empty) subset of vertices which admit exactly $j$ incoming edges in $E^*$, and write $n_{V_j}$ for the cardinality of $V_j$. Thus, we have 
\begin{align}\label{eq:nvnebound}
n_V = \sum_{j \geq 0} n_{V_j} \quad \text{ and } \quad n_E \geq \sum_{j \geq 0} j \cdot n_{V_j},
\end{align}
where the inequality follows from the fact that the sum of $j \cdot n_{V_j}$ over  $j \geq 0$ is the cardinality of $E^* \subset E$. The {\bf for} loop runs once per vertex of $V$, and the cost of each iteration is dominated by the equaliser computation in line 6. In the worst case, each equaliser computation requires Gaussian elimination on a matrix with $n_\bA$ rows (for $\bA_v$) and $2n_{V_0}n_\bA$ columns (for~$\Phi'_v$). For each vertex $v$ in $V_j$, there are $(j-1)$ such Gaussian eliminations to perform, so executing the {\bf for} loop for $v \in V_j$ incurs a cost of
$O((j-1)~n_{V_0}~n_\bA^3)$.
Since each $V_j$ contains $n_{V_j}$ vertices, the total cost of processing all vertices is given by
\[
O\left(\sum_{j \geq 0} n_{V_j}~(j-1)~n_{V_0}~n_\bA^3\right).
\]
From \eqref{eq:nvnebound}, we obtain $n_V \geq n_{V_0}$ and $n_E \geq\sum_j (j-1)~n_{V_j}$, which concludes the argument.
\end{proof}

\begin{proof}[Proof of Corollary~\ref{cor:overall_complexity}]
Summing the estimates from Lemma \ref{lem:acycredcomp} and Lemma \ref{lem:arbrepcomp} gives a total complexity of
$O(n_V^2+n_E+n_Vn_E)n_\bA^3)$.
The term $n_En_\bA^3$ is dominated by $n_Vn_En_\bA^3$ and may be omitted.
\end{proof}

\subsection{Examples} \label{ssec:secexamples}
We describe two instances where the space of sections $\Gamma(Q;\bA_\bullet)$ arises naturally. The first example is in the representation theory of finite groups.

\begin{example}\label{ex:grouprep} The ability to compute sections of quiver representations allows us to recover fixed spaces of group representations. Let $G$ be a finite group and $V$ a finite-dimensional vector space. A {\em representation} of $G$ valued in $V$ is a group homomorphism $\phi: G \to \text{GL}(V)$. Consider a quiver $Q$ with a single vertex $v$ and one edge $e(g)$ from $v$ to itself for each $g$ in~$G$. The data of our group representation produces a representation $\bA_\bullet$ of $Q$ whose vector space $\bA_v$ is $V$ and whose linear maps $\bA_{e(g)}$ are $\phi(g):V \to V$. As in Example \ref{ex:jordan}, the space of sections $\Gamma(Q;\bA_\bullet)$ is the {\bf fixed space} of $\phi$:
\[
\Gamma(Q;\bA_\bullet) = \set{v \in V \mid \phi(g)(v) = v \text{ for all } g \in G}.
\]
\end{example}

Every right Kan extension problem \cite[Chapter X]{maclane} for functors valued in the category of vector spaces can be solved by computing sections of an appropriate representation of a (possibly infinite) quiver. Our second example involves one such extension problem.

\begin{example} \label{ex:pushfwd}
Computing spaces of sections of quiver representations allows us to construct pushforwards of sheaves on posets. Given an order-preserving map $f:X \to Y$ between finite partially ordered sets, one can construct a pair of adjoint functors
\[
f_*:{\bf Sh}(X) \to {\bf Sh}(Y) \quad \text{ and } \quad f^*:{\bf Sh}(Y) \to {\bf Sh}(X)
\]
between the categories of sheaves (valued in finite-dimensional vector spaces, with respect to the Alexandrov topology) on $X$ and $Y$, see \cite[Sec 5]{curry2013sheaves} for details. Whereas the pullback~$f^*$ admits a straightforward definition, describing the pushforward $f_*\mathscr{S} \in {\bf Sh}(Y)$ of a sheaf $\mathscr{S} \in {\bf Sh}(X)$ is more delicate, since it requires computing the categorical limit
\[
f_*\mathscr{S}(y) = \lim_{f(x) \geq y} \mathscr{S}(x).
\]
Let $Q$ be the quiver with vertex set $X$ and edges  $x \to x'$ whenever $x \leq x'$. The sheaf $\mathscr{S}$ induces a representation $\bA_\bullet$ of $Q$, where $\bA_x$ is the stalk $\mathscr{S}(x)$ and the edge map $\bA_x \to \bA_{x'}$ is the restriction map $\mathscr{S}(x \leq x')$.
Given $y \in Y$, let $Q_{\geq y}$ be the quiver $Q$ restricted to the vertices $\set{x \in X \mid f(x) \geq y}$, and let $\bA^y_\bullet$ restrict the representation $\bA_\bullet$ to these vertices. The stalks of the desired pushforward coincide with the space of sections
\[
f_*\mathscr{S}(y) = \Gamma\left(Q_{\geq y};\bA^y_\bullet\right).
\]
\end{example}

\section{Principal Components via Optimisation}
\label{sec:PCs} 

Here we will define principal components with respect to a quiver representation as solutions to an optimisation problem over the space of sections. To this end, let us first recall the starting point, ordinary principal component analysis (PCA).

\begin{definition}\label{def:PCs}
Let $D := \set{y_1, \ldots, y_m}$ be a finite collection of mean-centred\footnote{i.e., $\frac{1}{m} \sum_i y_i$ lies at the origin} vectors in $\R^n$; the {\em sample covariance} of $D$ is the $n \times n$ symmetric matrix
\[
S:= \frac{1}{m} \sum_{i=1}^m y_i y_i\T,
\] where $\T$ indicates transpose. Assuming that the top $r$ eigenvalues $\lambda_1 > \cdots > \lambda_r$ of $S$ are distinct, the $r$-th {\bf principal component} $\PC_r(D)$ of $D$ is the $\lambda_r$-eigenspace of $S$. 
\end{definition}

Since the $r$-th principal component  is a one-dimensional subspace of $\R^n$, it is standard practice to represent it by any constituent nonzero vector in $\PC_r(D)$. Treating the sample covariance matrix as a bilinear form on $\R^n$ allows us to interpret principal components in terms of the following variance maximisation problem:  
\begin{equation}
    \label{eqn:PCs}
    \max_{X} \tr ( X\T S X ) \text{ subject to } X\T X = \text{id}_r.
\end{equation} 
Here $\tr$ indicates trace and $\text{id}_r$ is the $r \times r$ identity matrix. The columns of an optimal $n\times r$ matrix $X$ form an orthonormal basis for the space $\PC_{\leq r}(D)$ spanned by the top $r$ principal components, and solving~\eqref{eqn:PCs} for increasing $r$ gives the individual principal components in descending order. 

\subsection{Principal components along quiver representations} Consider a quiver $Q$ and fix a representation $\bA_\bullet$ of $Q$ valued in real vector spaces.
Henceforth we will fix an isomorphism $\R^{\dim \bA_v} \stackrel{\simeq}{\longrightarrow} \bA_v$ for each vertex $v$ in $V$, which allows us to impose (once and for all) an inner product structure on each $\bA_v$. Writing $n$ for the dimension of $\tot(\bA_\bullet)$,
\[
n = \sum_{v \in V} \dim \bA_v,
\]
we inherit an isomorphism $\R^n \stackrel{\simeq}{\longrightarrow} \tot(\bA_\bullet)$ and a concomitant inner product structure on the total space of $\bA_\bullet$. Making choices of ear decompositions and spanning arborescences for $Q$ produces a map $F:\R^d \to \tot(\bA_\bullet)$, described in \eqref{eq:F},
where $d = \dim \Gamma(Q;\bA_\bullet)$.  Expressed in terms of the chosen isomorphisms, $F$ becomes a full-rank $n \times d$ matrix whose image is an embedded copy of $\Gamma(Q;\bA_\bullet)$ inside $\R^n$. We are therefore able to define principal components relative to this embedding $F$. 
\begin{definition}
\label{def:qPCs}
Given any mean-centred finite subset $D$ of $\R^n \simeq \tot(\bA_\bullet)$, let $S$ be the sample covariance (as in Definition \ref{def:PCs}). For each $r \leq d$, consider the optimisation problem over all $n \times r$ matrices $X = [x_1 ~ x_2 ~ \cdots ~ x_r]$ prescribed by
\begin{equation}
    \label{eqn:qPCs}
    \max_{X} \tr ( X\T S X ) \quad \text{ subject to } \quad 
    \begin{array}{cc} X\T X = \text{id}_r  \text{ and } x_1,\ldots,x_r \in \Gamma(Q;\bA_\bullet). \end{array} 
\end{equation} 
The {\bf space of top $r$ principal components} of $D$ along $\bA_\bullet$ is the subspace $\PC_{\leq r}(D;\bA_\bullet)$ of $\R^n$ determined by the column span
\[
\PC_{\leq r}(D;\bA_\bullet) = \text{span}\set{x_1, \ldots, x_r}
\] of an optimal matrix $X$.
\end{definition}

It is possible to uniquely construct an optimal solution $X_*$ to \eqref{eqn:qPCs} by proceeding one column at a time and imposing the orthogonality of each column with respect to all of the preceding columns. The {\bf $r$-th principal component} of $D$ along $\bA_\bullet$ is the subspace $\PC_r(D;\bA_\bullet)$ spanned by the $r$-th column of $X_*$. In sharp contrast to the ordinary principal components from Definition \ref{def:PCs}, these principal components along $\bA_\bullet$ need not be eigenvectors of the covariance matrix $S$. There are, however, two special cases where ordinary principal components coincide with their quiver-compatible avatars.

\begin{proposition}
\label{prop:specials}
Assume that one of the two conditions below holds:
\begin{enumerate} 
\item either $D \subset \R^n$ lies entirely in the subspace $\Gamma(Q; \bA_\bullet)$, or
\item the edge set of $Q$ is empty.
\end{enumerate}
Then $\PC_r(D) = \PC_r(D;\bA_\bullet)$ for every $r \leq d$. 
\end{proposition}
\begin{proof}
If $D \subset \Gamma = \Gamma(Q;\bA_\bullet)$, then the sample covariance $S$ restricts to an endomorphism of $\Gamma$. For all $r \leq d$, the columns of any matrix $X$ that maximises~\eqref{eqn:PCs} must also lie in $\Gamma$. Thus, such an $X$ also maximises~\eqref{eqn:qPCs}. Finally, if there are no edges in $Q$ then $\Gamma$ equals all of $\R^n$ so \eqref{eqn:qPCs} reduces to \eqref{eqn:PCs}.
\end{proof}

In its most general form, linearly constrained PCA can be described as follows. The space of top $r$ principal components of $D \subset \R^n$ (with sample covariance $S$), constrained by some $n \times c$ matrix $W$, is the span of the columns of an optimal $n \times r$ matrix $X$ in
\begin{align*}
    \max_{X} \tr ( X\T S X ) \quad \text{ subject to } \quad 
    \begin{array}{ll} X\T X = \text{id}_r \text{ and } W\T X = 0 . \end{array} 
\end{align*}
This formulation follows from~\cite[Equation 7.4]{diamantaras1996principal}, and it is usually assumed that $W\T W = \text{id}_c$. Evidently, finding principal components along a quiver representation is a special instance of constrained PCA, provided we have access to an orthogonal basis for the complement of $\Gamma(Q;\bA_\bullet)$ in $\tot(\bA_\bullet)$. 

\subsection{Alternate perspectives}

Here we define two more optimisation problems related to~\eqref{eqn:qPCs}; as before, both will require a fixed choice of embedding $F : \R^d \to \tot(\bA_\bullet)$ of $\Gamma(Q;\bA_\bullet)$, where the map $F$ is viewed as an $n \times d$ matrix. Here is the first one, which is defined over the space over $d \times r$ matrices $Y$:
\begin{equation}
    \label{eqn:qPCs_param}
    \max_{Y} \tr( Y\T F\T S F Y)  \quad \text{subject to} \quad  Y\T (F\T F) Y = \text{id}_r.
\end{equation}
The $n \times n$ matrix $B:= FF\T$ serves as a (not necessarily orthogonal) projection onto the image of $F$. Now, we set $S_B:= BSB$ and consider another optimisation problem defined over $n \times r$ matrices $Z$: 
\begin{equation}
    \label{eqn:qPCs_proj}
    \max_{Z} \tr ( Z\T S_B Z) \quad \text{subject to} \quad Z\T (B^2) Z = \text{id}_r.
\end{equation}
Although the $r$ columns of $Z$ can be any $B^2$-orthonormal vectors in $\tot(\bA_\bullet)$, the optimal directions will lie in $\Gamma(Q;\bA_\bullet)$ because $S_B$ restricts to an endomorphism of $\Gamma(Q;\bA_\bullet)$ for any $r \leq d$. Our next result establishes the equivalence of these two alternate perspectives with the original one from Definition \ref{def:qPCs}. 

\begin{proposition}
\label{prop:3optimisation}
The maximum values of the three optimisation problems~\eqref{eqn:qPCs},~\eqref{eqn:qPCs_param}, and~\eqref{eqn:qPCs_proj} are all the same. Moreover, a matrix $X$ maximises~\eqref{eqn:qPCs} if and only if matrix $Y$ maximises~\eqref{eqn:qPCs_param} if and only if matrix $Z$ maximises~\eqref{eqn:qPCs_proj}, where 
\[
X = FY = BZ.
\]
\end{proposition}

\begin{proof}
We first show that $Z$ maximises~\eqref{eqn:qPCs_proj} if and only if $BZ$ maximizes \eqref{eqn:qPCs}. Since $\Gamma = \Gamma(Q;\bA_\bullet)$ is the image of $B$, it follows that the columns of $BZ$ all lie in $\Gamma$. Moreover, we have $(BZ)\T(BZ) = Z\T (B^2) Z = \text{id}_r$, so $BZ$ is orthogonal and satisfies all the constraints of~\eqref{eqn:qPCs}. Moreover, we have
\begin{align*}
\tr ( Z\T S_B Z) &= \tr (Z\T \cdot BSB \cdot Z)\\
&= \tr \left( (BZ)\T S (BZ) \right).
\end{align*} Conversely, given some $X$ maximising~\eqref{eqn:qPCs}, its columns $x_i$ are orthonormal vectors in $\Gamma$, hence $x_i = Bz_i$ for some $z_i \in \tot(\bA_\bullet)$. Letting $Z$ be the matrix of columns $z_i$ gives a solution to~\eqref{eqn:qPCs_proj} with the same maximal value (as confirmed by the trace calculation above). This gives the desired equivalence of~\eqref{eqn:qPCs} and~\eqref{eqn:qPCs_proj}. Turning now to \eqref{eqn:qPCs_param}, assume again that $Z$ maximises \eqref{eqn:qPCs_proj} and let $Y = F\T Z$, so
\[
(FY)\T (FY) = (BZ)\T (BZ) = \text{id}_r.
\] Computing the relevant trace for \eqref{eqn:qPCs_param} gives
\begin{align*}
\tr( Y\T F\T S F Y) &= \tr ( Z\T F F\T S F F\T Z) \\ &= \tr ( Z\T B SB Z) \\
&= \tr(Z\T S_B Z).
\end{align*}
Thus, the value of the objective function of~\eqref{eqn:qPCs_proj} at $Z$ equals the value of the objective function of \eqref{eqn:qPCs_param} at $Y = F\T Z$. Conversely, given some $Y$ maximising \eqref{eqn:qPCs_param}, its image $FY$ is a matrix of orthogonal vectors in $\Gamma$, hence lies in the feasible set for~\eqref{eqn:qPCs}, with the trace of $X = FU = BV$ in~\eqref{eqn:qPCs} being the same as the trace of $Y$ in~\eqref{eqn:qPCs_param}. 
\end{proof}

We consider \eqref{eqn:qPCs} an {\bf implicit} version of the optimisation problem to determine principal components along quiver representations, while \eqref{eqn:qPCs_param} and \eqref{eqn:qPCs_proj} are its {\bf parametrised} and {\bf projected} variants. Thanks to the preceding result, it becomes possible to freely translate between these three perspectives. In practice, the dimension $d$ of $\Gamma(Q;\bA_\bullet)$ is much smaller than the ambient dimension $n$ of $\tot(\bA_\bullet)$, so one might wish to work with the optimisation problem~\eqref{eqn:qPCs_param} in this smaller space. An algorithmic approach to~\eqref{eqn:qPCs_proj} that similarly reduces to a smaller space has been studied in~\cite{golub1973some}.

\begin{remark}
The argument invoked in the proof of Proposition \ref{prop:3optimisation} simplifies considerably if the $n \times d$ matrix $F$ has orthonormal columns. In this case,  the matrices $Y$ in~\eqref{eqn:qPCs_param} satisfy $Y\T Y = \text{id}_r$. Moreover, the matrix $Z$ that maximises~\eqref{eqn:qPCs_proj} satisfies $Z\T Z = \text{id}_r$. This is because $B = FF\T$ is an orthogonal projection onto $\Gamma$, so $v \in \Gamma$ if and only if $Bv =v$. Since the columns of $Z$ are in $\Gamma$ at the optimum, we have $BZ = Z$ and hence $\text{id}_r = Z\T B\T B Z = Z\T Z$, as claimed.
\end{remark}

\subsection{Examples}

We conclude this section with some examples to illustrate principal components along quiver representations. As for usual principal components, they give a low-dimensional projection of the data, with interpretable coordinates, in which features may be found. We first consider a statistically motivated example.

\begin{example}
\label{ex:marginals} 
Consider the quiver representation $\RR^2 \leftarrow \RR^4 \to 
\RR^2$ with arrow maps 
$$ A = \begin{bmatrix} 1 & 1 & 0 & 0 \\ 0 & 0 & 1 & 1 \end{bmatrix} \quad \text{and} \quad B = \begin{bmatrix} 1 & 0 & 1 & 0 \\ 0 & 1 & 0 & 1 \end{bmatrix} , $$
cf. Example~\ref{ex:two_arrow}.
The space of sections is the image of $\RR^4$ under the flow map, i.e. the points
$$ x = \begin{bmatrix} x_{11} & x_{12} & x_{21} & x_{22} & x_{1+} & x_{2+} & x_{+1} & x_{+2} \end{bmatrix} \in \tot(\bA_\bullet) \simeq \RR^8,$$
where $+$ denotes summing over an index, e.g. $x_{1+} = x_{11} + x_{12}$. 
The first four coordinates $x_{ij}$ are joint observations, and the last four coordinates $x_{i+}$ and $x_{+j}$ are two pairs of marginal observations. For example, if $x_{ij}$ is a gene expression measurement for gene $i$ in cell type $j$, then $x_{i+}$ sums over the two cell types while $x_{+j}$ sums over the two genes.

The principal components along the quiver representation are directions in $\RR^8$ that maximise the variance in the data, subject to taking the form of a joint observation and its two marginals.
Note that one could also consider the principal components of the joint observations in~$\RR^4$, and take their image under the flow map $F$ to give directions in $\RR^8$. 
These directions will in general not coincide with the principal components along the quiver representation, cf. the last paragraph of Example~\ref{ex:pc_onearrow}. 
\end{example} 

We next consider a biological setting, involving gene expression measurements.

\begin{example}
The development of methods that relate bulk and single cell transcriptomics data is an active area of study, see e.g.~\cite[Figure 1]{jew2020accurate}.
Bulk RNA-seq gives an average gene expression across the cells in a sample. 
Single cell RNA-seq gives measurements for each cell. Then, cells can be clustered to give the a gene expression value for each cell type.

Fixing $g$ genes and $c$ cell types, we consider the quiver representation
$ \RR^{g c} \longrightarrow \RR^g $.
The linear map on the arrow is $I \otimes a \in \RR^{g \times gc}$, where the $c$ cell types are assumed to be present in proportions $a = (a_1, \ldots, a_c) \in \RR^{1 \times c}$. 
The sample data lies in $\tot(\bA_\bullet) = \RR^{g c} \times \RR^g$. For a sample $(v, w) \in \RR^{gc} \times \RR^g$, the entry $v_{ij}$ is the gene expression of gene $i$ in cell type $j$, while $w_k$ is the bulk measurement for gene $k$. 
The data will in general not lie in the space of sections, on account of of the different measurement techniques, as well as variation in the proportions of cell types.

The principal components along the quiver representation are directions in $\RR^{gc} \times \RR^g$ that exhibit high variance in the data while being consistent between the single cell and bulk measurements. Given a coarser assignments of cells into types, this idea extends to the quiver representation $\RR^{gc} \to \RR^{gt} \to \RR^g$, where $t < c$ is the number of cells types in the coarser clustering, see e.g.~\cite[Figures 3 and 4]{wolf2019paga}. If the two assignments of cells into types are not compatible, we instead consider the quiver representation $\RR^{gc} \to \RR^{g} \leftarrow \RR^{gt}$.
\end{example}

\section{Principal Components as Generalised Eigenvectors} \label{sec:gsvd}

We have already noted that -- aside from some very special cases as in Proposition \ref{prop:specials} -- the principal components $\PC_r(D;\bA_\bullet)$ of Definition \ref{def:qPCs} are not eigenvectors of the sample covariance $S$. Here we remedy this defect by providing a spectral interpretation for $\PC_r(D;\bA_\bullet)$. All scalars, vectors and matrices described below live over the field of real numbers.

\begin{definition}
\label{def:gevp}
Fix two identically-sized square matrices $A$ and $B$. The {\bf generalised eigenvalues} of the matrix pencil $A - \lambda B$ are the solutions $\lambda$ to $\det(A - \lambda B) = 0$. We call a non-zero vector $x$ with $Ax = \lambda Bx$ a {\bf generalised eigenvector}, with $\lambda$ its generalised eigenvalue. 
\end{definition}

Our main tool in the quest to interpret quiver prinicipal components as generalised eigenvectors is the {generalised singular value decomposition} (GSVD) \cite[Theorem 2]{van1976generalizing}.

\begin{theorem}
\label{thm:gsvd} {\em [GSVD]}
Given positive integers $a \geq b \geq c$, fix an $(a \times c)$ matrix $A$ and a $(b \times c)$ matrix $B$. There exist 
\begin{enumerate}
    \item orthogonal matrices $W_A$ and $W_B$ of size $a \times a$ and $b \times b$ respectively,
   \item (rectangular) diagonal matrices $\Delta$ and $\Sigma$ of size $a \times c$ and $b \times c$ respectively, and
    \item a $c \times c$ invertible matrix $G$,
\end{enumerate}
satisfying both
\[
 A = W_A \Delta G \qquad \text{and} \qquad B = W_B \Sigma G.
 \]
 \end{theorem}

The matrices $W_A, W_B$ and $G$ are not uniquely determined, but {the ratios $\delta_i^2/\sigma_i^2$} of the {squares of the} diagonal entries of $\Delta$ and $\Sigma$ are completely specified (up to reordering) by $A$ and $B$. We note en passant that a different generalisation of the singular value decomposition \cite[Theorem 3]{van1976generalizing} also appears in the context of constrained PCA, and that a discussion of GSVD naming conventions can be found in~\cite[Section 5.5]{takane2001constrained}. Returning to the setting of interest, we fix a representation $\bA_\bullet$ of a quiver $Q$ and select a full-rank $n \times d$ matrix $F:\R^d \to \tot(\bA_\bullet)$ whose image is $\Gamma(Q;\bA_\bullet)$. The following result is Theorem (B) from the Introduction.

\begin{theorem}
\label{thm:obtain_qPCs}
Let $S$ be the sample covariance of a sufficiently generic mean-centered subset $D = \set{y_1,\ldots,y_m} \subset \R^n$ of cardinality $m \geq n$. For each $r \leq d$, the $r$-th principal component $\PC_r(D;\bA_\bullet)$ is spanned by $Fu_r$, where $u_r$ is the eigenvector of the matrix pencil $F\T S F - \lambda (F\T F)$ corresponding to its $r$-th largest generalised eigenvalue. 
\end{theorem}
\begin{proof}

Let $M$ denote the $m \times n$ matrix whose $i$-th row is the normalised vector $y_i/\sqrt{m}$, so that the sample covariance satisfies $S = M\T M$. Noting that $m \geq n \geq d$, we apply the GSVD from Theorem \ref{thm:gsvd} to the $m \times d$ matrix $A = MF$ and the $n \times d$ matrix $B = F$. This produces factorisations
\[
MF = W_A \Delta G \qquad \text{and} \qquad F = W_B \Sigma G
\]
with orthogonal $W_A, W_B$, invertible $G$, and diagonal $\Delta,\Sigma$. Since $D$ is generic and $F$ has full rank, we may safely assume that the diagonal entries of $\Delta$ and $\Sigma$ are nonzero. And by orthogonality of both the $W$-matrices, we obtain two new identities
\begin{align}\label{eq:gsvdMF}
(MF)\T(MF) = G\T\Delta^2G \qquad \text{and} \qquad F\T F = G\T\Sigma^2G.
\end{align}
Since $S = M\T M$ by design, the first identity reduces to $F\T S F = G\T\Delta^2G$. Let us write $\set{\delta_1,\ldots,\delta_d}$ and $\set{\sigma_1,\ldots,\sigma_d}$ for the (necessarily nonzero) diagonal entries of $\Delta$ and $\Sigma$ respectively, and denote by $g_i$ the $i$-th column of $G^{-1}$. It follows from \eqref{eq:gsvdMF} that $g_i$ is a generalised eigenvector for the $d \times d$ matrix pencil $(F\T S F) - \lambda \cdot (F\T F)$, corresponding to the generalised eigenvalue $\lambda_i := \nicefrac{\delta_i^2}{\sigma_i^2}$. In other words, we have
\begin{align}\label{eq:Gcols}
(F\T S F) g_i = \lambda_i \cdot (F\T F) g_i.
\end{align}

The top $d \times d$ block $\Sigma_d$ of $\Sigma$ is invertible because its diagonal has nonzero entries. Since $G$ is also invertible, the product $\Sigma_dG$ permutes the set of $d \times r$ matrices via $Y \mapsto Y_\circ = \Sigma_dGY$, which allows us to re-express the optimisation \eqref{eqn:qPCs_param} in a particularly convenient form. To this end, we calculate:
\begin{align*}
Y\T (F\T S F) Y &= Y\T (G\T\Delta^2G) Y & \text{by \eqref{eq:gsvdMF}} \\
&= (G^{-1}\Sigma_d^{-1}Y_\circ)\T (G\T\Delta^2G) (G^{-1}\Sigma_d^{-1}Y_\circ) & \text{since }Y_\circ = \Sigma_dGY \\
&= Y_\circ\T~\Sigma_d^{-1}\Delta^2\Sigma_d^{-1}~Y_\circ & \text{after two cancellations}.
\end{align*}
Now the intermediate product $\nabla := \Sigma_d^{-1}\Delta^2\Sigma_d^{-1}$ is a $d \times d$ diagonal matrix whose $i$-th diagonal entry is $\lambda_i = \nicefrac{\delta^2_i}{\sigma^2_i}$. Reordering basis vectors if necessary, we can assume without loss of generality that $\lambda_1 > \cdots > \lambda_d$. The change of variables $Y \mapsto Y_\circ$ transforms the optimisation problem from \eqref{eqn:qPCs_param} into
\[
\max_{Y_\circ} \tr ( Y_\circ\T \nabla Y_\circ ) \quad \text{subject to} \quad Y_\circ\T Y_\circ = \text{id}_r.
\]
This is the ordinary PCA optimisation \eqref{eqn:PCs}, which {generically} admits a unique solution $Y_*$ obtained by successively increasing $r$. Since $\nabla$ is diagonal, the $i$-th column of $Y_*$ is the $i$-th elementary basis vector. Thus, the columns $\set{u_1, \ldots, u_r}$ of $U = G^{-1}\Sigma_d^{-1}Y_*$ lie in the directions of the corresponding columns of $G^{-1}$. By \eqref{eq:Gcols}, these columns are generalised eigenvectors associated to the $r$ largest generalised eigenvalues of our matrix pencil. Finally, applying $F$ to $U$ gives the principal components along the quiver representation as in Proposition~\ref{prop:3optimisation}. \end{proof} 

It follows that the top principal component is $Fu$, where $u$ maximises the Rayleigh quotient
\begin{equation}
    \label{eqn:rayleigh}
    \frac{u\T \! (F\T \! S F) u}{u\T \! (F\T F)u}, 
\end{equation} 
but in general for $r>1$ the optimisation~\eqref{eqn:qPCs_param} is not equivalent to a single trace ratio problem (see \cite{ngo2012trace}). 

\begin{remark}
Since the embedding $F:\RR^d \to \tot(\bA_\bullet)$ has rank $d$, the $d \times d$ matrix $F\T F$ is invertible. We can therefore convert the generalised eigenproblem of Theorem \ref{thm:obtain_qPCs} into the usual eigenvector problem $F^+ SF x = \lambda x$, where $F^+ = (F\T F)^{-1} F\T$ is the pseudo-inverse. However, as explained in~\cite[Section 4]{golub1973some}, it is often preferable to work with the generalised eigenvalue problem as the matrix $F^+ SF$ may not be symmetric. And depending on the condition number of $(F\T F)^{-1}$, the numerical stability might be worse .
\end{remark} 

\begin{example}
\label{ex:pc_onearrow}
Consider the quiver 
\begin{tikzcd} 
u \bullet \arrow[r,"e"] & \bullet v
  \end{tikzcd}
with representation:
 \begin{center}
  \qquad \qquad
 \begin{tikzcd} 
\RR^p \, \bullet \arrow[r,"J"] & \bullet \, \RR^q . \end{tikzcd} \end{center} 
Writing $n=p+q$ for the dimension of the total space, the $n \times n$ sample covariance $S$ of some $D \subset \R^n$ and the embedding $F:\R^{p} \to \R^n$ can be written as
\[
S = \begin{bmatrix} S_{uu} & S_{uv} \\ S_{vu} & S_{vv} \end{bmatrix}, \qquad F = \begin{bmatrix} \text{id}_p \\ J 
\end{bmatrix},
\]
where $S_{vu} = S_{uv}\T$.
Theorem~\ref{thm:obtain_qPCs} shows that the principal components are given by the generalised eigenvectors of the matrix pencil $A-\lambda B$ spanned by 
\[
A = S_{uu} + J\T S_{vu} + S_{uv}J + J\T S_{vv}J, \qquad B = \text{id}_p + J\T J . 
\]
In the special case where $D$ lies in the image of $F$, we have
\[
S = \begin{bmatrix} \text{id}_p \\ J \end{bmatrix} S_{uu} \begin{bmatrix} \text{id}_p & J\T \end{bmatrix} = \begin{bmatrix} S_{uu} & S_{uu} J\T \\ J S_{uu} & J S_{uu} J\T \end{bmatrix}, 
\]
and the matrix pencil is spanned by 
\[
A = S_{uu} + J\T J S_{uu} + S_{uu} J\T J + J\T J S_{uu} J\T J, \qquad  B = \text{id}_p + J\T J.
\]

If, in addition, $J\T J$ equals $\eta \text{id}_p$ for some scalar $\eta$, then this specialises further to give the matrix pencil spanned by $A = (1 + 2 \eta + \eta^2)S_{uu}$ and $B = (1 + \eta)\text{id}_p$. Now the principal components along the quiver representation are given by $F\xi$, where $\xi$ are the usual principal components of $D$ restricted to the vector space $\RR^p$ on the first vertex of the quiver.
\end{example}

\section{Learning Quiver Representations} \label{sec:learn}

We conclude this paper with a discussion focused on the problem of learning quiver representations from observed data. Fix a quiver $Q = (s,t:E \to V)$, and assume that we have full knowledge of the real vector spaces $\set{\bA_v \mid v \in V}$ assigned by some $Q$-representation $\bA_\bullet$ to all the vertices. However, none of the linear maps $\bA_e : \bA_{s(e)} \to \bA_{t(e)}$ are known. Instead, we are given access to mean-centred data $\set{y_1, \ldots, y_m}$, where each $y_i$ is a vector in the total space $\tot(\bA_\bullet) \simeq \RR^n$. Our task is to determine the $\bA_e$ maps that best fit the available data; here we will show how in special cases this task reduces to well-studied problems. It will be convenient to define, for each vertex $v$, the $m \times \dim \bA_v$ matrix $Y_v$ whose $i$-th row is the part of $y_i$ that lies in $\bA_v$.

\begin{example}
\label{ex:onearrow}
Consider the quiver 
\begin{tikzcd} 
u \bullet \arrow[r,"e"] & \bullet v
  \end{tikzcd}
with representation
 \begin{tikzcd} 
\bA_u \, \bullet \arrow[r,"\bA_e"] & \bullet \, \bA_v , \end{tikzcd} with matrix $\bA_e$ unknown.
Given data $y_i = (y_{i,u}, y_{i,v}) \in \bA_u \times \bA_v$ for $i \in \set{1, \ldots, m}$, minimising the Euclidean distance between $y_{i,v}$ and $\bA_e y_{i,u}$ for each $i$ gives the least squares optimisation problem 
 \[
 \min_{\bA_e} \| Y_v - Y_u \bA_e\T \|.
 \] Thus, the optimal estimate for $\bA_e\T$ is $(Y_u)^+ Y_v$, where $(Y_u)^+$ indicates the Moore-Penrose inverse of $Y_u$.
\end{example}

The preceding example can equivalently be viewed as training (or, learning the $\dim(\bA_u) \times \dim(\bA_v)$ parameters in) a linear neural network with full bipartite connections between a single input and output layer:
\begin{center}
    \begin{tikzpicture}
\path node (v2) at (0.5,3) []{$\dim(\bA_u) \, \Bigg\{ 
$};

\path node (v2) at (2,3) []{$\vdots$};
\path node (v1) at (2,3.7) {$\circ$};
\path node (v1a) at (2,2.3) {$\circ$};

\path node (w2) at (4,3) []{$\vdots$};
\path node (w1) at (4,3.5) {$\circ$};
\path node (w1a) at (4,2.5) {$\circ$};

\path node (x2) at (5.5,3) []{$\Big\} \dim(\bA_v)
$};
\draw [->] (v1) -- (w1); 
\draw [->] (v1) -- (w1a); 
\draw [->] (v1a) -- (w1); 
\draw [->] (v1a) -- (w1a); 
\end{tikzpicture}
\end{center}
The principal components along the quiver representation are then pairs of points in the input space $\bA_u$ and the output space $\bA_v$ that fit the weights on the edges and along which high variance is seen in the data.

\begin{remark}
More generally, a linear neural network with $k$ layers corresponds to learning a quiver representation on the quiver with $k$ edges 
\[
	\xymatrixcolsep{.6in}
	\xymatrix{
		\bullet \ar@{->}[r]^{e_1} & \bullet \ar@{->}[r]^{e_2} & \cdots \ar@{->}[r]^{e_{k-1}} & \bullet \ar@{->}[r]^{e_k} & \bullet
	}
	\]
	Each vertex $v$ is replaced by $\dim(\bA_v)$ scalar nodes, with full bipartite connections between nodes in adjacent layers. 
A more general architecture could involve other quivers.
For example, loops arise from lateral interactions~\cite[Figure 5]{baldi1995learning}.
The setting of learning parameters in linear neural networks with two layers is itself closely connected to principal component analysis \cite{baldi1989neural}.
\end{remark}

One way to extend the above to more general quivers is to learn the map on each edge $e$ independently, which amounts to minimising the objective function: 
\begin{equation}
    \label{eqn:optimiseA} \sum_{e \in E} \left(  \| Y_{t(e)} - Y_{s(e)} \bA_e\T \|^2 \right).
\end{equation}
Now the estimate for each edge map $\bA_e$ is given by Example~\ref{ex:onearrow}. If the quiver $Q$ is an arborescence, then the optimisation \eqref{eqn:optimiseA} falls into the setting of a {\em Gaussian graphical model} \cite{lauritzen1996graphical,sullivant2018algebraic} associated to a certain directed acyclic graph with $\dim \tot(\bA_\bullet)$ vertices, as we now describe. 

\begin{definition}\label{def:quivblowup}
Let $\delta:V \to \mathbb{Z}_{\geq 0}$ be a function from the vertices of an arborescence $Q$ to the non-negative integers. The {\bf $\delta$-blowup} of $Q$ is the quiver $Q_\delta$ where each $v \in V$ is replaced by $\delta(v)$ vertices, and each edge $e \in E$ is replaced by a complete directed bipartite graph whose edges go from the $\delta(s(e))$ vertices replacing $s(e)$ to the $\delta(t(e))$ vertices replacing $t(e)$.
\end{definition}

The directed acyclic graph of interest to us here is the $\delta$-blowup of the arborescence $Q$ where $\delta(v) = \dim \bA_v$. We denote this blowup by $Q_{\dim (\bA_\bullet)}$. For instance, if $Q$ is the arborescence on the left and  $\bA_\bullet$ is the representation (known only on the vertices) depicted in the middle, then the blowup $Q_{\dim (\bA_\bullet)}$ is shown to the right.

\begin{minipage}{0.6\textwidth}
 \begin{center} 
 \begin{tikzcd}
\bullet \arrow[r] & \bullet \arrow[rd]\arrow[r] & \bullet \\ 
 & & \bullet \\ 
  \end{tikzcd}
  \qquad 
\begin{tikzcd} 
\RR \arrow[r] & \RR^{2} \arrow[rd]\arrow[r] & \RR^{3} \\ 
 & & \RR^{2} \\ 
  \end{tikzcd}
\end{center}
\end{minipage}\begin{minipage}{0.3\textwidth}
\begin{center} 
  \begin{tikzpicture}
\path node (v1) at (2,3) {$\circ$};
\path node (v3) at (6,3) []{$\circ$};
\path node (v3a) at (6,2.5) []{$\circ$};
\path node (v3b) at (6,2) []{$\circ$};
\path node (v4) at (6,1) []{$\circ$};
\path node (v4a) at (6,0.5) []{$\circ$};
\path node (w1) at (4,3) []{$\circ$};
\path node (w1a) at (4,2.5) []{$\circ$};
\draw [->] (w1) -- (v4); 
\draw [->] (w1) -- (v4a); 
\draw [->] (w1a) -- (v4); 
\draw [->] (w1a) -- (v4a); 
\draw [->] (v1) -- (w1); 
\draw [->] (v1) -- (w1a); 
\draw [->] (w1) -- (v3); 
\draw [->] (w1) -- (v3a); 
\draw [->] (w1) -- (v3b); 
\draw [->] (w1a) -- (v3); 
\draw [->] (w1a) -- (v3a); 
\draw [->] (w1a) -- (v3b);
\end{tikzpicture}
\end{center}
\end{minipage}

\noindent The entries of the unknown matrices $\bA_e$ become unknown scalar weights on the edges of $Q_{\dim (\bA_\bullet)}$. Maximum likelihood estimation in the Gaussian graphical model learns the weights on these edges by minimising least squares error. Since this is equivalent to \eqref{eqn:optimiseA}, it gives an identical estimate for the unknown maps in the quiver representation. 

Although Definition \ref{def:quivblowup}  extends verbatim to the case where $Q$ is not an arborescence, the maximal likelihood estimation strategy described above is restricted to the setting of an arborescence. This is because 
weights of incoming edges at a vertex of the directed acyclic graph are summed over in a graphical model~\cite[Equation (13.2.3)]{sullivant2018algebraic}. By comparison, in the quiver setting we do not sum incoming edges from different vertices of the quiver in~\eqref{eqn:optimiseA}. 
Thus the above strategy only works when each vertex in the quiver has at most one incoming edge. 

The local assumption governing the choice of objective function in~\eqref{eqn:optimiseA} is that the maps $\bA_e$ can be learned independently of one another; this does not take into account the goodness of fit of data along longer paths in the quiver. Given such a path $p$, one may wish to minimise the distance between $y_{i,t(p)}$ and $\bA_p y_{i,s(p)}$. This yields an immediate generalisation of the objective function \eqref{eqn:optimiseA}, where one sums the contributions of each path in $Q$, rather than just over each edge. For acyclic quivers, such an optimisation can be approached by using a suitable partial order on edges, but it is more complicated for quivers with cycles. We defer a more general study of learning maps in quiver representations to future work. 

Our final example is an illustration of finding principal components along a learned quiver representation. This combines parameter estimation with principal component analysis, as is also seen in~\cite{wiesel2009decomposable,meng2012distributed,tang2021integrated}. 

\begin{example}
Consider once again the quiver with one edge $e: u \to v$ and representation $\RR^p \to \RR^q$ with unknown $\bA_e$. The best estimate is given by $(Y_u^+ Y_v)\T$, as described in Example~\ref{ex:onearrow}. Thus, a parameterisation of the space of sections $\Gamma(Q;\bA_\bullet)$ is given by
\[
F = \begin{bmatrix} I \\  (Y_u^+ Y_v)\T \end{bmatrix}.
\]
The top principal component along the quiver representation is the direction in the image of $F$ along which there is maximum variance in the data. This can be computed using Theorem~\ref{thm:obtain_qPCs} via the matrix pencil from Example~\ref{ex:pc_onearrow}, provided that we set $J = (Y_u^+Y_v)\T$. 
\end{example}

	\bibliographystyle{alpha}
	\bibliography{references}
	
\end{document}